\documentclass[12pt]{article}
\usepackage{amssymb}
\usepackage{amsmath,acronym,bm,natbib,graphics,mathrsfs,graphicx,epsfig,subfigure,placeins,indentfirst,setspace,enumerate,enumitem,caption,varioref,booktabs,tabularx,color,epstopdf,multirow}
\setcitestyle{authoryear,open={(},close={)}}
\usepackage{amsthm}
\makeatletter \oddsidemargin  -.1in \evensidemargin -.1in
\begin{document}
	\newcommand{\bea}{\begin{eqnarray}}
		\newcommand{\eea}{\end{eqnarray}}
	\newcommand{\nn}{\nonumber}
	\newcommand{\bee}{\begin{eqnarray*}}
		\newcommand{\eee}{\end{eqnarray*}}
	\newcommand{\lb}{\label}
	\newcommand{\nii}{\noindent}
	\newcommand{\ii}{\indent}
	\newtheorem{theorem}{Theorem}[section]
	\newtheorem{example}{Example}[section]
	\newtheorem{corollary}{Corollary}[section]
	\newtheorem{definition}{Definition}[section]
	\newtheorem{lemma}{Lemma}[section]
	\newtheorem{remark}{Remark}[section]
	\newtheorem{proposition}{Proposition}[section]
	\numberwithin{equation}{section}
	\renewcommand{\qedsymbol}{\rule{0.7em}{0.7em}}
	\renewcommand{\theequation}{\thesection.\arabic{equation}}
	\renewcommand\bibfont{\fontsize{10}{12}\selectfont}
	\setlength{\bibsep}{0.0pt}
		\title{\bf Weighted (residual) varentropy  and its applications**}
	
\author{ Shital {\bf Saha}\thanks {Email address: shitalmath@gmail.com,~520MA2012@nitrkl.ac.in} ~and  Suchandan {\bf  Kayal}\thanks {Email address (corresponding author):
		kayals@nitrkl.ac.in,~suchandan.kayal@gmail.com 
			\newline**It has been accepted on \textbf{Journal of Computational and Applied Mathematics}.}
	\\{\it \small Department of Mathematics, National Institute of
		Technology Rourkela, Rourkela-769008, India}}
\date{}
\maketitle
		\begin{center}
Abstract
		\end{center}
		In information theory, it is of recent interest to study variability of the uncertainty measures. In this regard, the concept of varentropy has been introduced and studied by several authors in recent past. In this communication, we study the weighted varentropy and weighted residual varentropy. Several theoretical results of these variability measures such as the effect under monotonic transformations and bounds are investigated. Importance of the weighted residual varentropy over the residual varentropy is presented. Further, we study weighted varentropy for coherent systems and weighted residual varentropy for proportional hazard rate models.  A kernel-based non-parametric estimator for the weighted residual varentropy is also proposed. The estimation method is illustrated using simulated and two real data sets.
		 \\
		 \\
		\textbf{MSCs:} 62N05; 60E05; 94A17.

			\section*{Abbreviations}
				\begin{acronym}
					\acro{PMF:}{Probability mass function} 
					\acro{SE:}{Shannon entropy}
					\acro{PDF:}{Probability density function} 
					\acro{WSE:}{ Weighted Shannon entropy}
					\acro{CDF:}{Cumulative distribution function} 
					\acro{WRSE:}{Weighted residual Shannon entropy}
					\acro{VE:}{Varentropy}
					\acro{WVE:}{Weighted varentropy}
					\acro{WRVE:}{Weighted residual varentropy}
					\acro{MRL:}{Mean residual lifetime}
					\acro{VRL:}{Variance residual lifetime}
					\acro{PHR:}{Proportional hazard rate}
			\end{acronym}
			
\section{Introduction}
The concept of entropy, introduced by \cite{shannon1948mathematical} has been widely used in computing chaos and uncertainty of a system. For a discrete type random variable $X$ with probability mass function (PMF) $P(X=x_i)=p_i$, where $i=1,\ldots,n$, the Shannon entropy (SE) is defined as 
\begin{eqnarray}\label{eq1.1}
S(X)=-\sum_{i=1}^{n}p_i\log p_i,
\end{eqnarray}
where `log' denotes the natural logarithm. Equation (\ref{eq1.1}) quantifies the information contained in the random variable $X$. Clearly, $S(X)$ is a function of probabilities of happening of events only.  There are some situations, where it is necessary to consider both these probabilities and some qualitative characteristics of the events in order to quantify uncertainty contained in a random experiment. For example, in a two-handed game, a player needs to think about both the probabilities of different random strategies and the wins corresponding to these strategies. To compute the information contained in a probabilistic experiment  having elementary events characterized both by their probabilities and by some qualitative weights, \cite{guiacsu1971weighted} proposed weighted entropy of a discrete type random variable $X$, which is given by  (see also \cite{guiasu1986grouping})
\begin{eqnarray}\label{eq1.2}
S^{w}(X)=-\sum_{i=1}^{n}w_i p_i\log p_i,
\end{eqnarray}
where $w_i\ge0$ is known as the weight corresponding to the $i$th elementary event. The weight is usually specified by the analyst depending upon the situation.

For a non-negative absolutely continuous random variable $X$ with probability density function (PDF) $f(\cdot)$, the SE and weighted Shannon entropy (WSE) are given by 
\begin{eqnarray}\label{eq1.3}
H(X)=-\int_{0}^{\infty}f(x)\log f(x)dx~\mbox{and}~H^x(X)=-\int_{0}^{\infty}xf(x)\log f(x)dx,
\end{eqnarray}
respectively. In the literature, $H(X)$ and $H^x(X)$ are also known as the differential entropy and weighted differential entropy, respectively. The factor $x$ in the integrand of  $H^x(X)$ given in (\ref{eq1.3}) may be treated as a weight linearly emphasizing the occurrence of the event $\{X = x\}$. It is worthwhile to mention that one may assume a general weight function $w(x)\ge0$ instead of $x$ inside the integration of $H^{x}(X)$ in (\ref{eq1.3}). Note that the differential entropy is a shift-independent measure, whereas the weighted differential entropy is shift-dependent. This property makes $H^x(X)$ a more useful tool than $H(X)$ in some situations. For details about the weighted differential entropy, we refer to \cite{di2007weighted}.

Denote by $IC(X)=-\log f(X)$, the random information content of a non-negative absolutely continuous  random variable $X$ with PDF $f(\cdot).$ Then, $H(X)=E[IC(X)],$ that is, the SE in (\ref{eq1.3}) is the expectation of the information content of  $X.$ We note that the justification of  $IC(X)$ can be given in discrete case, where it represents the number of bits essentially required to represent $X$ by a coding scheme that minimizes average code length (see \cite{shannon1948mathematical}). For continuous case, one may still call $IC(X)$ as information content, even though, the coding interpretation does not hold (see \cite{bobkov2011concentration}). The differential entropy takes values from the extended real line, whereas the SE of a discrete random variable is always non-negative. Some other deficiencies of the differential entropy can be found in \cite{schroeder2004alternative}.
The notion of differential entropy has been widely applied in stochastic modelling and various applied fields. Due to importance of the information content in information theory, probability and statistics, it is of interest to study its behaviour. \cite{bobkov2011concentration} studied that the information content concentrates around the entropy in high dimension when the PDF is log-concave. Sometimes, two random variables have the same SE. The Shannon entropies of the uniform distribution in $(0,1)$ and  the exponential distribution with mean $1/e$ are equal, and they are zero. In this case, the  concept of concentration of information content around entropy is useful for analysis purposes. More concentration of information content around entropy will yield a better result. The concentration can be calculated with the variance of $IC(X)$, which is called varentropy. Recently, the varentropy has attracted the attention of the computer scientists for measuring the dispersion of sources, and has many applications in data compression, for details reader may refer to \cite{kontoyiannis1997second} and \cite{kontoyiannis2013optimal}. Suppose $X$ is a non-negative and absolutely continuous random variable with PDF $f(\cdot)$, then the varentropy is given by (see \cite{fradelizi2016optimal}) 
\begin{eqnarray}\label{eq1.4}
{\it VE}(X)=Var(-\log f(X))=\int_{0}^{\infty}f(x)(\log f(x))^2dx-\bigg(\int_{0}^{\infty}f(x)\log f(x)dx\bigg)^2.
\end{eqnarray}
In a similar way, the varentropy of a discrete random variable $X,$ which assumes values $x_i$, $i=1,\ldots,n$ can be defined as (see \cite{di2021analysis}) 
\begin{eqnarray}\label{eq1.5}
{\it VE}^*(X)
&=&\sum_{i=1}^{n}p_i[\log p_i]^2-\bigg[\sum_{i=1}^{n}p_i\log p_i\bigg]^2,
\end{eqnarray}
where $P(X=x_i)=p_i$, $i=1,\ldots,n.$ From (\ref{eq1.4}), it is clear that the varentropy of $X$ is always non-negative. For uniform distribution, it is zero. The varentropy indicates how the information content is scattered around the entropy. Several researchers studied varentropy and discussed its properties including some characterization results. In this direction, we refer to \cite{bobkov2011concentration},  \cite{arikan2016varentropy}, \cite{fradelizi2016optimal}, \cite{di2021analysis} and \cite{maadani2022varentropy}. In particular, \cite{bobkov2011concentration} demonstrated a concentration property of the information content $IC(X)$ when an $n$-dimensional random vector has log-concave density function. \cite{arikan2016varentropy} established that the sum of the varentropies at the output of the polar transform is less than or equal to the sum of the varentropies at the input, with equality if and only if at least one of the inputs has zero varentropy. \cite{fradelizi2016optimal} provided bound for the varentropy of random vectors with log-concave densities, which is sharper than that proposed by \cite{bobkov2011concentration}. \cite{di2021analysis} introduced residual varentropy and studied its various properties with some applications related to the proportional hazards model and the first-passage times of an Ornstein-Uhlenbeck jump-diffusion process. Recently, \cite{maadani2022varentropy} studied varentropy for order statistics. The authors introduced a new stochastic order based on the varentropy and relationships of it with the other stochastic orders. Besides these, one may also refer to \cite{buono2022varentropy}, \cite{raqab2022varentropy} and \cite{sharma2023varentropy}  for some results on varentropy. 

Unlike differential entropy, the weighted differential entropy given in (\ref{eq1.3}) can be expressed as $H^{x}(X)=E(-X\log f(X)).$ Denote by 
\begin{align}
IC^{w}(X)=-w(X)\log f(X),
\end{align}
for a given choice of weight $w(x)=x.$ We call $IC^{x}(X)$ as weighted random information content. Thus, the weighted differential entropy is the expectation of the weighted information content of $X.$ In statistics, one may think of the weighted information content as the weighted log-likelihood function. In this paper, we introduce weighted varentropy (WVE) of a random variable $X$ and study its properties. We will later observe that the WVE is not affine invariant, whereas, the varentropy is affine invariant. That is, the varentropy is unable to distinguish variability in the information content of a random process in different time points. This drawback can be resolved if we use WVE. Note that the use of WVE is also motivated by the necessity, arising in various communication and transmission problems, of expressing the usefulness of events with a variance measure. Further, an important characteristic of the human visual system is that it can recognize objects in a scale and translation invariant manner. We refer to \cite{wallis1996using} for details. It is a challenging job to achieve this desirable behavior using biologically realistic networks. Indeed, knowing that a device fails to operate, or a neuron fails to release spikes in a given time-interval, yields a relevantly different information from the case when such an event occurs in a different equally wide intervals. In some cases, we are thus led to resort to a shift-dependent information/variance measure that, for instance, assigns different measures to such distributions. The concept of WVE can be employed in such situations for a better result than the usual varentropy.

The paper is organized as follows. In Section $2$, we introduce WVE and discuss some properties. The closed-form expressions of the WVE of some distributions have been obtained. The WVE of an affine transformed random variable has been provided. Various bounds of the WVE have been presented.  In Section $3$, we propose  WRVE and present some results. Various bounds of the WRVE have been obtained. The WRVE has been studied under general monotonic transformations. Further, we present a justification of studying WRVE. A comparative study between the WRVE and WRE has been provided for exponential and uniform distributions.  Section $4$ deals with the WVE of coherent systems. Some bounds are proposed. In Section $5$, we study  WRVE for the  proportional hazard rate models. Section $6$ focuses on a non-parametric estimator of the WRVE. We consider a simulated data set and two real life data sets for the purpose of illustration of the proposed estimation method.  Finally, Section $7$ concludes the paper. 

Throughout the rest of the paper, we assume that random variables are non-negative and absolutely continuous. The derivatives are assumed to exist whenever they are used. The notation `$\prime$'  is used to denote derivative, for example $h'(x)=\frac{d}{dx}h(x)$.

\section{Weighted varentropy}
In this section, we introduce WVE and discuss its various properties.  Let $X$ be a non-negative  absolutely continuous random variable with PDF $f(\cdot)$.  The variance of the weighted information content of $X,$ denoted by $IC^{w}(X)=-w(X)\log f(X)$ with weight function $w(\cdot)$ is defined as 
\begin{eqnarray}\label{eq2.1}
{\it VE}^w(X)=Var[IC^w(X)]
&=& E[(IC^w(X))^2]-[E(IC^w(X))]^2 \nonumber\\
&= & \int_{0}^{\infty}w^2(x)f(x)[\log f(x)]^2dx-\bigg[\int_{0}^{\infty}w(x)f(x)\log f(x)dx\bigg]^2.\nonumber\\
\end{eqnarray}
From the expressions of $H(X)$ in (\ref{eq1.3}) and ${\it VE}(X)$ in (\ref{eq1.4}), it is clear that the SE and varentropy do not depend on the realizations of the random variable $X$, but depend only on its PDF. This does not hold for the WSE and WVE. From (\ref{eq1.3}) and (\ref{eq2.1}), it is easy to observe that the WSE and WVE depend on both the PDF and realizations of the random variable.  For a discrete random variable $X$ taking values $x_{i}$, $i=1,\ldots,n$, the WVE is given by
 \begin{eqnarray}\label{eq2.2}
{\it VE}^{*^w}(X)
&=&\sum_{i=1}^{n}w_i^2p_i[\log p_i]^2-\bigg[\sum_{i=1}^{n}w_i p_i\log p_i\bigg]^2,
\end{eqnarray}
where $w_i\ge 0$ is the weight corresponding to the event $\{X=x_i\}$.   The WVE in (\ref{eq2.2}) measures the variability of the information content supplied by a probabilistic experiment depending both on the probabilities of events and on qualitative (objective or subjective) weights of the possible events. The qualitative weights are usually the weights in physics or as the utilities in decision theory. Even though, assigning a weight to every elementary event is not an easy job to be done, we generally assign larger weight to the events which are more significant or more useful depending on our goal. Indeed, one may assign non-negative numbers (weights) to each probabilistic events, directly proportional to its importance or significance. Based on the values of the weights, the WVE reduces to different variability measures. For example, if we choose the weights to be $1$, then the WVE reduces to the varentropy given in (\ref{eq1.5}). Further, let us consider the weight as the ratio of the probabilities of the events and the information supplied by the events with a scaling factor $\frac{1}{2}$, given by
\begin{eqnarray}\label{eq2.3}
w_{i}=-\frac{p_{i}}{2\log p_{i}}.
\end{eqnarray}
Then, the WVE in (\ref{eq2.2}) reduces to the varextropy of a discrete type random variable $X$, given by (see \cite{vaselabadi2021results})
\begin{eqnarray}
{\it VEx}(X)=\frac{1}{4}\left[\sum_{i=1}^{n}p_{i}^{3}-\left(\sum_{i=1}^{n}p_{i}^{2}\right)^2\right].
\end{eqnarray}
We refer to \cite{guiacsu1971weighted} for an elaborate discussion regarding the importance and some choices of the weights for weighted entropy.  From (\ref{eq1.5}), it can be easily established that the varentropy for a discrete random variable is permutation symmetric with respect to the probabilities $(p_1,\ldots,p_n)$. This property does not hold for the WVE given in (\ref{eq2.2}). Now, we consider an example, illustrating the fact that the WVE is not permutation symmetric with respect to the assigned probabilities. 
\begin{example}\label{ex2.1}
	Consider a discrete type random variable $X\in\{1,2,3\}$ with 
	$$P(X=1)=p,~~P(X=2)=1-p-q,~~\mbox{and}~~P(X=3)=q,$$ 
	where $0\leq q\leq1-p\leq1$. Using (\ref{eq1.2})  and (\ref{eq2.2}), and assuming $w_i=x_i,$ we obtain
	\begin{eqnarray}
		S^x(X)=-x_1p\log p-x_2(1-p-q)\log (1-p-q)-x_3q\log q
	\end{eqnarray}
	and 
	\begin{eqnarray}\label{eq2.4}
		{\it VE}^{*^x}(X)&=&x_1^2p(\log p)^2+x_2^2(1-p-q)(\log (1-p-q))^2+x_3^2q(\log q)^2-[S^x(X)]^2\nonumber\\
		&=&{\it VE}^{*^x}(X;p,q),~\mbox{say},
	\end{eqnarray}
	where $x_1=1,x_2=2$ and $x_3=3.$ Now, consider two different sets of probabilities $\{p=0.5,1-p-q=0.2,q=0.3\}$ and $\{p=0.3,1-p-q=0.2,q=0.5\}$ for two different distributions of $X.$ Here, ${\it VE}^*(X;p,q)=0.13296441046={\it VE}^*(X;q,p)$, that is the varentropies of two different distributions are same. But, their WVEs are not same. Here, ${\it VE}^{*^x}(X;p,q)=1.92508326678$ and ${\it VE}^{*^x}(X;q,p)=0.48838794637$, implying that WVEs are not permutation symmetric.  The surface plots (upward and downward views) of the WVE given in (\ref{eq2.4}) are presented in Figure $1.$
		 	\begin{figure}[h!]
		\begin{center}
			\subfigure[]{\label{c1}\includegraphics[height=1.9in]{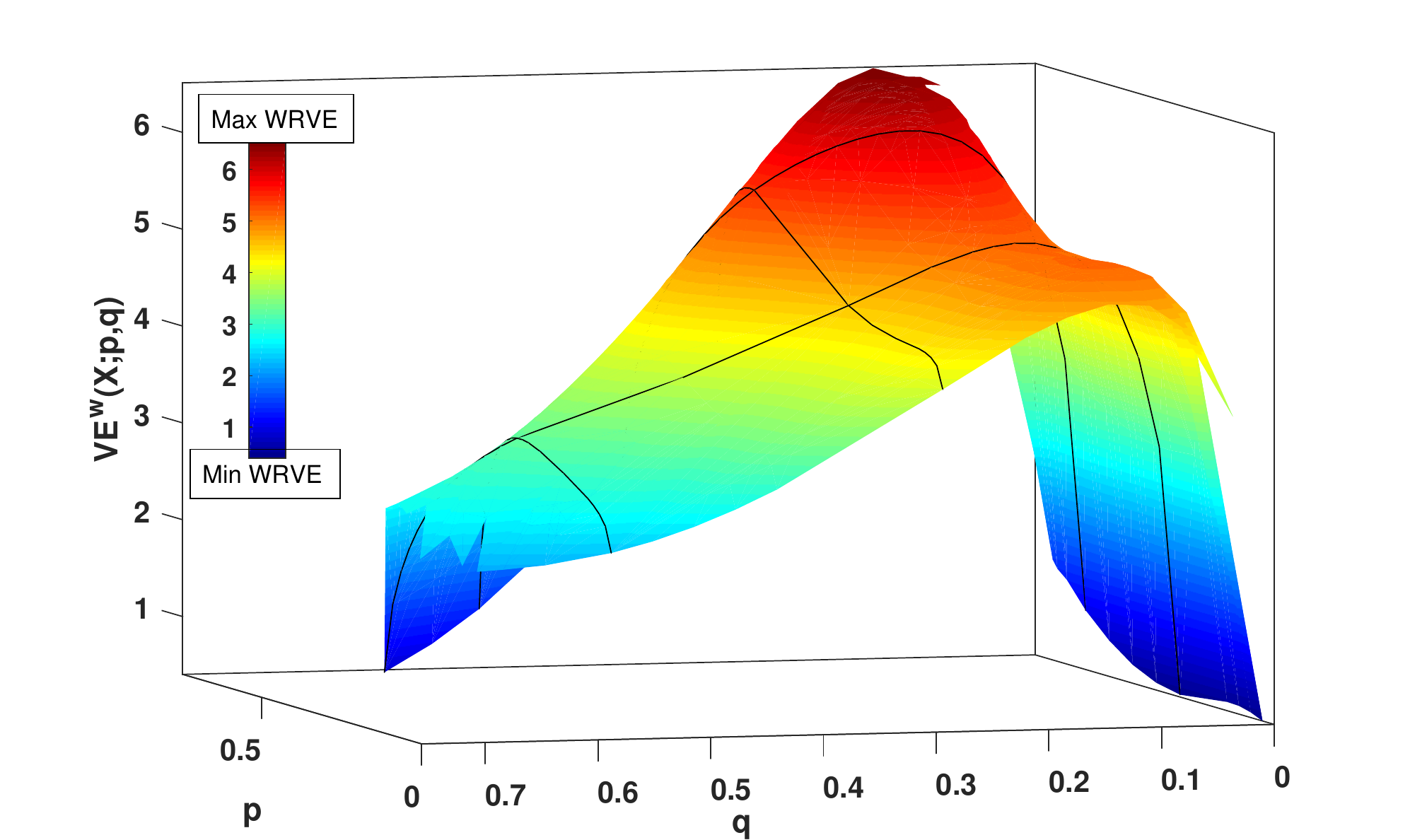}}
			\subfigure[]{\label{c1}\includegraphics[height=1.9in]{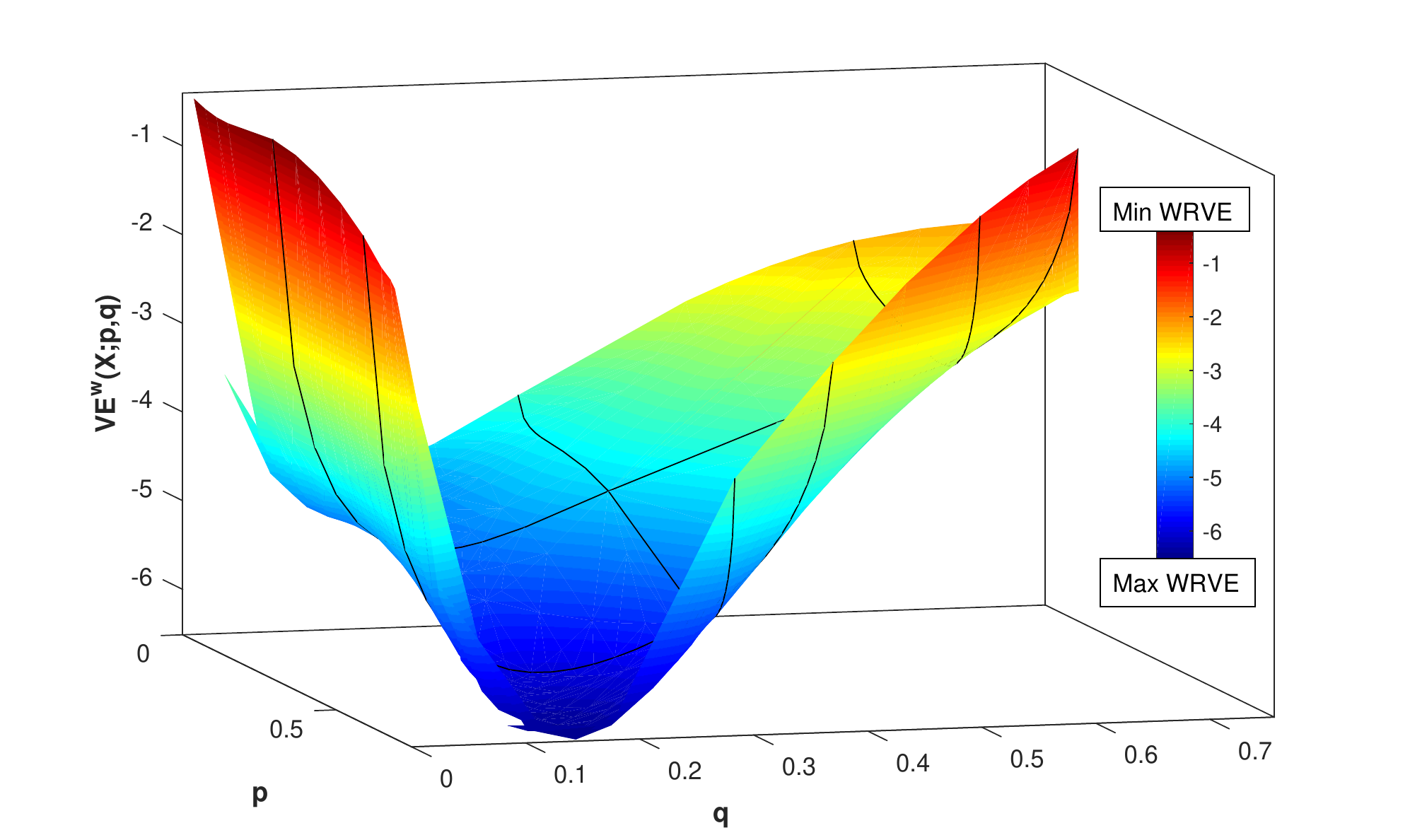}}
			\caption{$(a)$ Upward and $(b)$ downward views of the WVE given in (\ref{eq2.4}) for the distribution in Example \ref{ex2.1}.}
		\end{center}
	\end{figure}
\end{example}
In the next example, we compute WVE given in (\ref{eq2.1}) with $w(x)=x$ for some continuous distributions. 

 \begin{example}\label{ex2.2}~~
	\begin{itemize}
		\item[(i)] Let $X$ follow uniform distribution in the interval $(a,b)$. The WVE is obtained as 
		$$ {\it VE}^x(X)=\frac{1}{12}(a-b)^2[\log (b-a)]^2,$$
		which is depicted in Figure $2(a)$  by keeping the values of $a$ to be fixed and $b$ variable.  Further, note that  for the case of uniformly distributed random variable in the interval $(0,1)$, the WVE is $0.$ From the graph, we observe that when the parameters $a$ and $b$ are close to each other,  as expected, the WVE tends to zero. 
		
		 	\begin{figure}[h!]
			\begin{center}
				\subfigure[]{\label{c1}\includegraphics[height=1.9in]{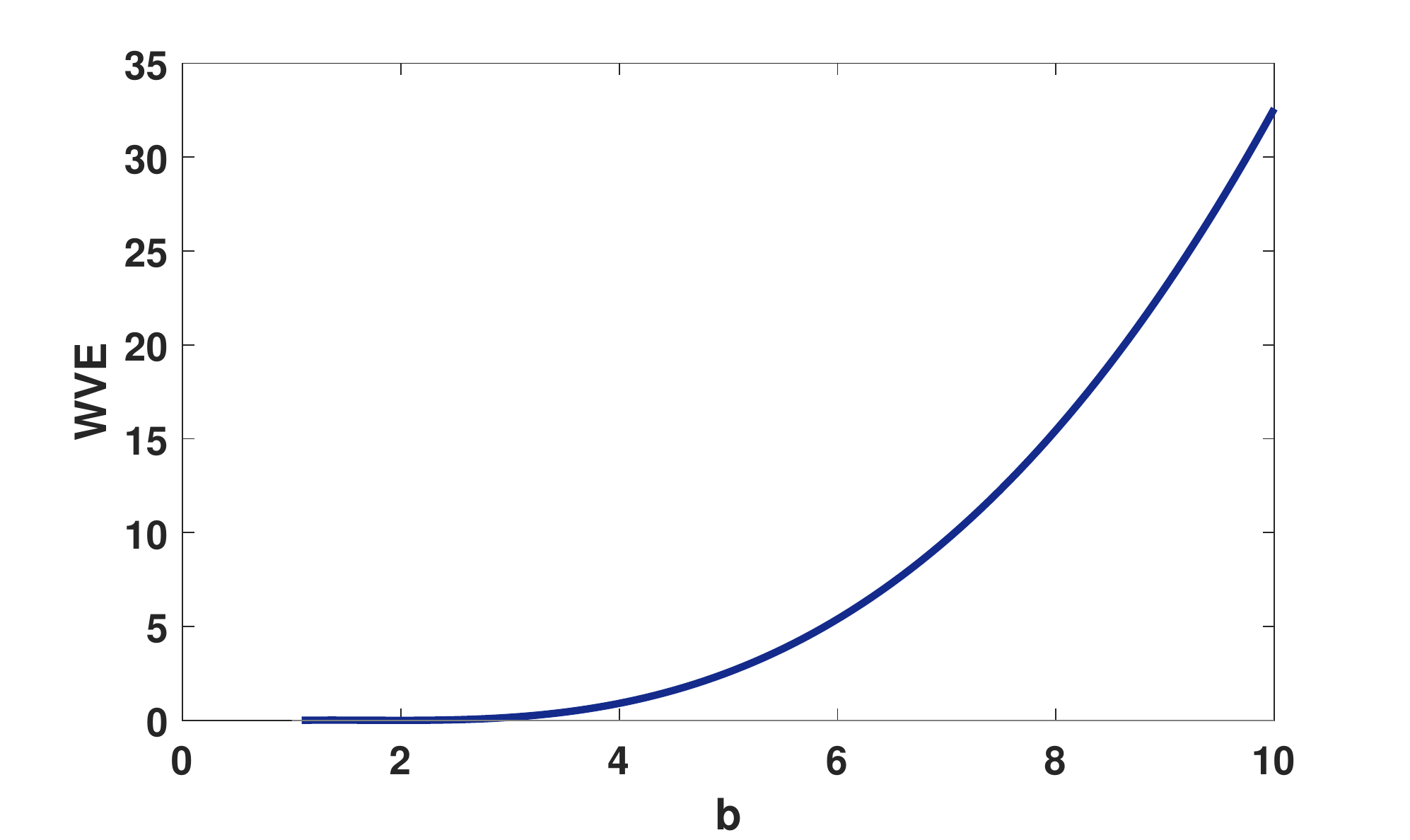}}
				\subfigure[]{\label{c1}\includegraphics[height=1.9in]{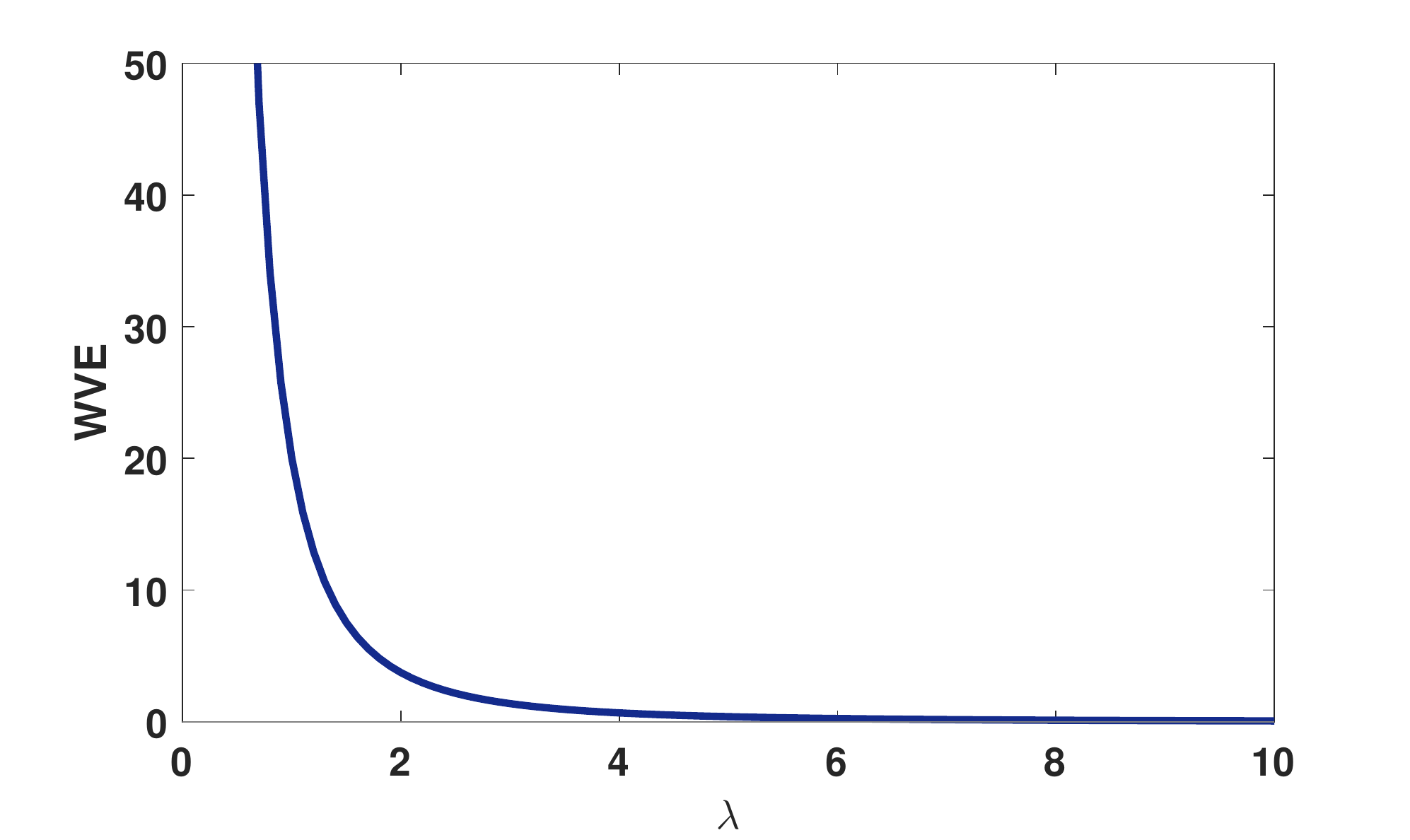}}
				\subfigure[]{\label{c1}\includegraphics[height=1.9in]{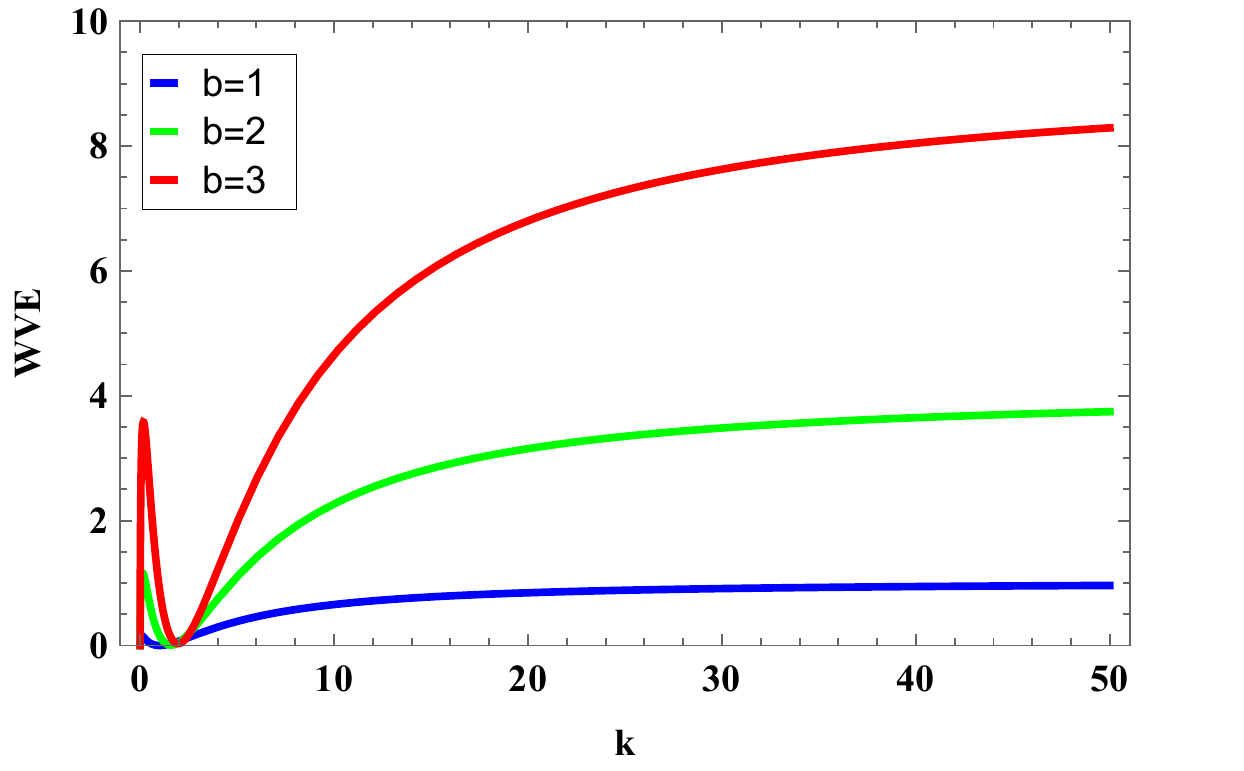}}
				\subfigure[]{\label{c1}\includegraphics[height=1.9in]{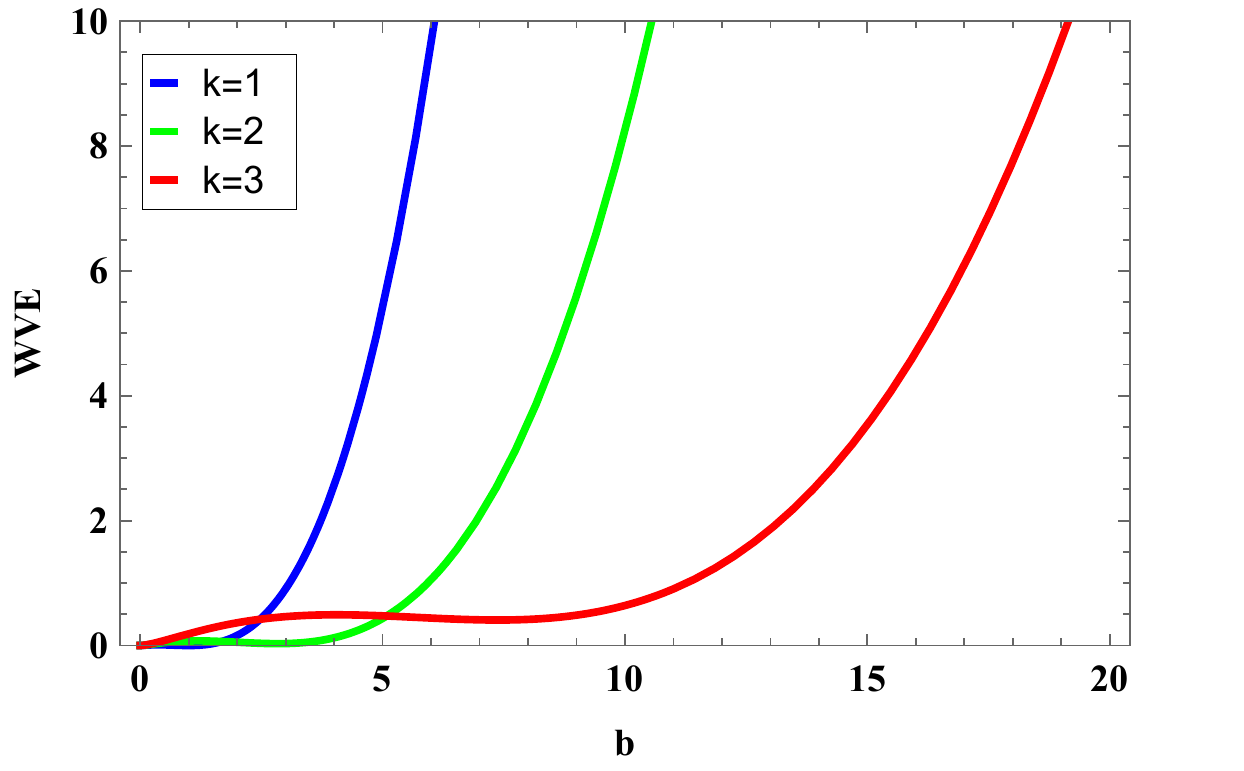}}
				\caption{Graph of the WVE for $(a)$ uniform distribution with respect to $b$ as in Example \ref{ex2.2}$(i)$ for $a=1$, $(b)$  exponential distribution with respect to $\lambda$ as in Example \ref{ex2.2}$(ii)$, $(c)$  the power distribution  with respect to $k$,  for $b=1,2,3$ as in Example \ref{ex2.2}$(iii)$, and $(d)$ the power distribution with respect to $b$ for $k=1,2,3$ as in Example \ref{ex2.2}$(iii)$.}
		\end{center}
		\end{figure}
		
		\item[(ii)] Suppose $X$ follows exponential distribution with positive mean $1/\lambda.$ The WVE is 
		$$ {\it VE}^x(X)=\frac{1}{\lambda^2}[(\log \lambda-4)^2+4],$$
		which is plotted in Figure $2(b)$. From this figure, we observe that WVE is close to zero, when $\lambda$ takes large values.  Here, $ {\it VE}^x(X)\lesssim 0.1201,$ for $\lambda>8$.
		
		\item[(iii)] Let $X$ follow power distribution with CDF $F(x)=(\frac{x}{b})^k,~0<x<b$, $k>0$, and $b>0$. The WVE is obtained as
		{\begin{eqnarray*}
		\it VE^x(X)&=&
		\frac{kb^2}{(k+2)^3}[(k+2)^2(\log(k/b^k))^2+2(k-1)(k+2)\log(k/b^k)\{(k+2)\log(b)\\
		&~&~~~-1\}+(k-1)^2\{(k+2)^2(\log(b))^2-2\big((k+2)\log(b)-1\big)\}]\\
		&~&~~~~+\frac{(kb)^2}{(k+1)^2}[(k+1)\log(k/b^k)+(k-1)\{(k+1)\log (b)-1\}]^2.
\end{eqnarray*}}
	which is plotted in Figures $2(c)$ and $2(d)$ with respect to $k$ and $b$, for different values of $b$ and $k$, respectively. The graphs show that the WVE of the power distribution is not monotone with respect to the model parameters.
	\end{itemize}
\end{example}

It is not always an easy task to get closed-form expression of the WVE of a random variable of interest. In this case, certain transformation can be used to compute WVE of a transformed random variable (new distribution) in terms of the WVE of the known distribution. The following result is useful in this regard.

 \begin{proposition}
	Suppose $X$ is a random variable with PDF $f(\cdot)$ and $Y=aX+b$, with $a>0$ and $b\geq0$. Then, the WVE of $Y$ is obtained as
	\begin{eqnarray*}
	{\it VE}^y(Y)={\it VE}^{w_1}(X)+(a\log a)^2Var(X)+2\log a (H^{w_1^2}(X)-H^{w_1}(X)E[aX+b]),
	\end{eqnarray*}
	where $H^{w_1}(X)$ and ${\it VE}^{w_1}(X)$ are the WSE and WVE of $X$, respectively for a given choice of  $w_{1}(x)=ax+b$.
\end{proposition}

\begin{proof}
	Denote by $g(.)$ the PDF of $Y=aX+b$, where  $g(y)=\frac{1}{a}f(\frac{y-b}{a}),~ y>b.$ From (\ref{eq2.1}), we have 
	\begin{eqnarray}\label{eq2.5}
	{\it VE}^y(Y)=\int_{0}^{\infty}y^2g(y)[\log g(y)]^2dy-[H^y(Y)]^2.
	\end{eqnarray}
	Further, 
	\begin{eqnarray}\label{eq2.6}
	H^{y}(Y)=-\int_{0}^{\infty}yg(y)\log g(y)dy&=&-\int_{b}^{\infty}y\frac{1}{a}f\left(\frac{y-b}{a}\right)\log \left(\frac{1}{a}f\left(\frac{y-b}{a}\right)\right)dy\nonumber\\
	&=&-\int_{0}^{\infty}(ax+b)f(x)\{\log f(x)-\log a \}dx\nonumber\\
	&=&H^{w_1}(X)+\log a ~E[aX+b]
	\end{eqnarray}
	and 
	\begin{eqnarray}\label{eq2.7}
	\int_{0}^{\infty}y^2g(y)[\log g(y)]^2dy
	&=&\int_{0}^{\infty}(ax+b)^2f(x)\{\log f(x)-\log a \}^2dx\nonumber\\
	&=&\int_{0}^{\infty}(ax+b)^2f(x)\{\log f(x) \}^2dx+2\log a H^{w_1^2}(X)\nonumber\\
	&~&+(\log a)^2E[(aX+b)^2].
	\end{eqnarray}
	After using (\ref{eq2.6}) and (\ref{eq2.7}) in (\ref{eq2.5}), we get the required result. Thus, the proof is completed. 
\end{proof}

Sometimes, it is not always possible to obtain the closed-form expression of WVE. In such cases, bounds are useful for some ideas about the variability measure.  In the next proposition, we obtain an upper bound of the WVE defined in (\ref{eq2.1}) with weight $w(x)=x$. We note that here, the bound depends on the WSE. 
\begin{proposition}\label{prop2.2}
		Let $X$ be a non-negative and absolutely continuous random variable with PDF $f(\cdot)$ satisfying
		\begin{eqnarray}\label{eq2.10}
		e^{-(\alpha x+\beta )}\leq f(x)\leq1,~~~x>0,
		\end{eqnarray}
		with $\alpha>0$ and $\beta\geq0$. Then,
		$$VE^x(X)\leq  H^{w_{2}}(X),$$
		where $H^{w_{2}}(X)$ is the WSE with weight function $w_{2}(x)=\alpha x^3+\beta x^2.$
	\end{proposition}
	
	\begin{proof} 
		Under the assumptions made, from (\ref{eq2.1}), after some simplification we obtain
		\begin{eqnarray*}
			VE^w(X)&\leq &-\int_{0}^{\infty}x^2f(x)(\alpha x+\beta)\log f(x)dx\\
			&=& -\int_{0}^{\infty}(\alpha x^3+\beta x^2)f(x)\log f(x)dx\\
			&=& H^{w_{2}}(X).
		\end{eqnarray*}
		Thus, the result follows.
\end{proof}

The following example shows that the Lomax distribution satisfies the condition in Proposition \ref{prop2.2}.

\begin{example}\label{ex2.3}
	Let $X$ follow Lomax distribution with CDF $F(x)=1-(1+\frac{x}{a})^{-b},~x>0,~a>0,~b>0$. Then, it can be established graphically (see Figure $3$) that the condition $e^{-(\alpha x+\beta)}\leq f(x)\leq1,~x>0$ holds for some choices of the values of $\alpha$, $\beta$, $a$, and $b$.
\end{example}

\begin{figure}[htbp!]\label{fig3}
\centering
\includegraphics[width=14cm,height=8cm]{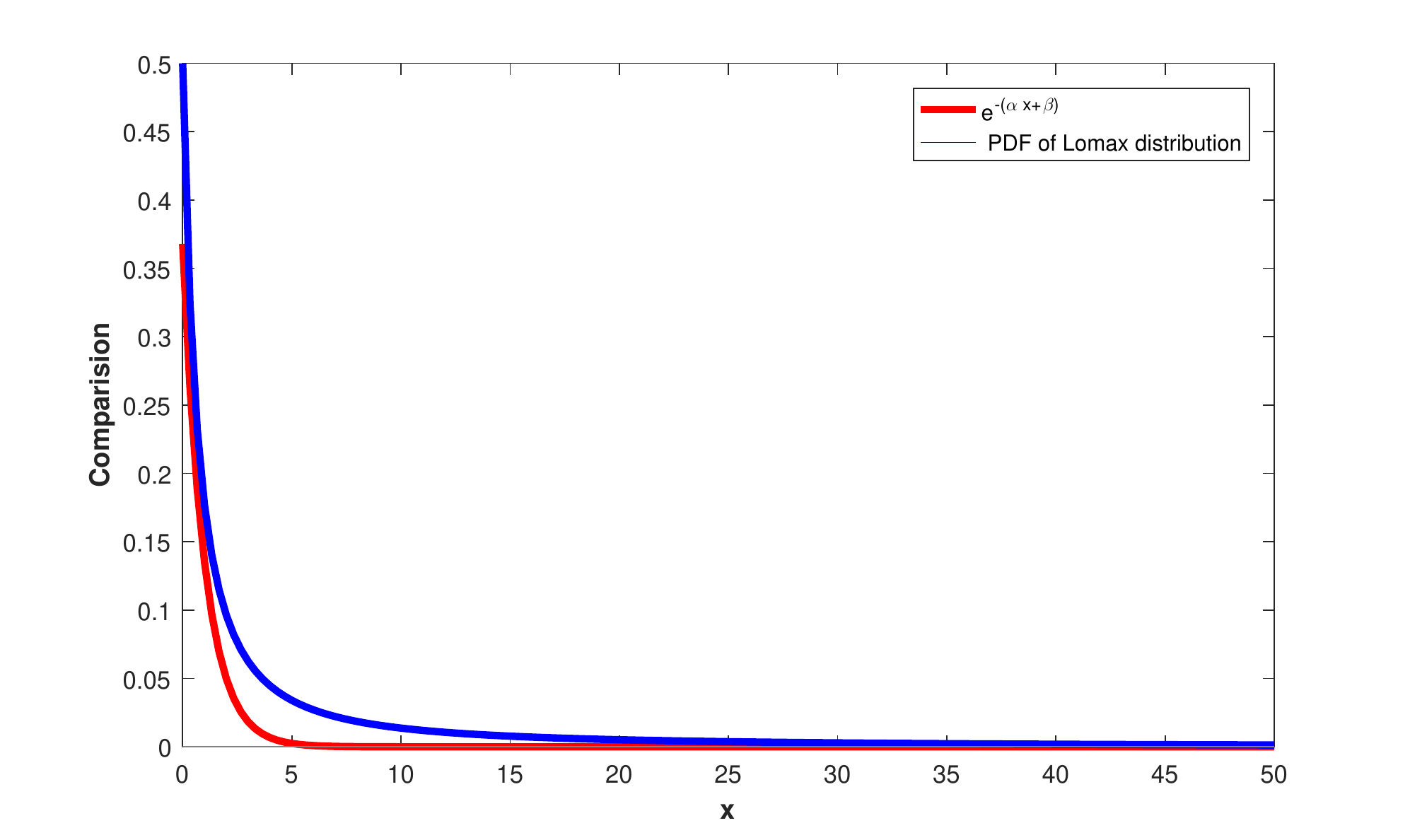}
\caption{Graphs of $e^{-(\alpha x+\beta)}$ for $\alpha=1$, $\beta=0.5$ and the PDF of Lomax distribution with $a=1$ and $b=1$ as in Example \ref{ex2.3}}.
\label{fig:PdfofIGLFRD*}
\end{figure}

 Next, we investigate the WVE under monotonic transformations, and look into its characteristic features. Denote $$\eta_1(x)=\phi(x)\log\phi'(x)~ \mbox{and}~  \eta_2(x)=\phi(x)\log(-\phi'(x)),~x>0.$$
\begin{theorem}\label{th2.2}
	Let $X$ be a non-negative random variable with PDF $f(\cdot)$. Suppose $\phi(\cdot)$  is strictly monotone, continuous, and differentiable function; and $Y=\phi(X)$. Then,
	\begin{equation}\label{eq2.8}
{\it VE}^y(Y)=\left\{
	\begin{array}{ll}
	{\it VE}^{\phi}(X)+Var[\eta_1(X)]-2E[\phi(X)\eta_{1}(X) \log f(X)]
	-2H^{\phi}(X)E[\eta_1(X)],\\\text{if  $\phi$  is  strictly  increasing;}
	\\
	\\
	-{\it VE}^{\phi}(X)-Var[\eta_2(X)\big)]+2E[\phi(X)\eta_2(X) \log f(X)]-2\big(H^{\phi}(X)\big)^2\\
	-2H^{\phi}(X)E[\eta_2(X)]-2\{E[\eta_2(X)]\}^2,
	~ \text{if  $\phi$  is  strictly  decreasing,}
	\end{array}
	\right.
	\end{equation}
	where
	$H^{\phi}(X)$ is the WSE with weight $w(x)=\phi(x)$ and  ${\it VE}^{\phi}(X)$ is given in (\ref{eq2.1}) with $w(x)=\phi(x)$.
\end{theorem}
\begin{proof} 
	First, we consider that $\phi(\cdot)$ is strictly increasing. Then, from (\ref{eq2.1}) the WVE of $Y=\phi(X)$ is expressed as
	\begin{eqnarray}\label{eq2.9}
	{\it VE}^y(Y)&=&\int_{0}^{\infty}y^2 f(\phi^{-1}(y))\bigg|\frac{1}{\phi'(\phi^{-1}(y))}\bigg|\bigg[\log f(\phi^{-1}(y))-\log\phi'(\phi^{-1}(y))\bigg]^2dy\nonumber\\
	&~& -\Bigg[\int_{0}^{\infty}y f(\phi^{-1}(y))\bigg|\frac{1}{\phi'(\phi^{-1}(y))}\bigg|\bigg[\log  f(\phi^{-1}(y))-\log\phi'(\phi^{-1}(y))\bigg]dy\Bigg]^2.\nonumber
	\\
	\end{eqnarray}
	Now, substituting $y=\phi(x)$ in the right hand side of (\ref{eq2.9}), and then after some simplification, the desired result follows. The proof when $\phi(\cdot)$ is strictly decreasing is similar, and thus  it is omitted for brevity.				  
\end{proof}

The following remark which is immediate from Theorem \ref{th2.2}, provides the effect of WVE under the scale and location transformations.
\begin{remark}~
	\begin{itemize}
		\item[(a)] Let $Y=aX,~a>0$ be the scale transformation. Then,
		\begin{eqnarray*}
			{\it VE}^{ax}(aX)&=&a^2\big\{{\it VE}^x(X)+(\log a)^2Var(X)-2\log a\big(H^x(X)E(X)
			\\&~&+E[X^2\log f(X)]\big)\big\}.
		\end{eqnarray*}
		\item[(b)] Let $Y=X+b,~b\ge0$ be the location transformation. Then,
		\begin{eqnarray*}
			{\it VE}^{x+b}(X+b)={\it VE}^x(X)-b{\it VE}(X),
		\end{eqnarray*}
		where ${\it VE}(X)$ is the varentropy of $X$, given by (\ref{eq1.4}).
	\end{itemize}
\end{remark}

\cite{fradelizi2016optimal} proved that for the distributions with log-concave PDF, the varentropy is smaller than $1$. The condition of log-concave density means that the PDF belongs to the class of strong unimodal densities. Further, \cite{di2021analysis} established a similar result for the residual varentropy when a random lifetime X is ILR, i.e. its PDF is log-concave. Recently, \cite{di2021stochastic} obtained the similar results if $X_t$ is IFR (increasing failure rate)
and its PDF is strictly decreasing. Thus, it is natural to investigate if the WVE given by (\ref{eq2.1}) is less than $1$ for the case of a log-concave density function. Next, we consider an example, which establishes that the WVE does not have such upper bound like varentropy and residual varentropy.
\begin{example}
	Consider Weibull distribution with PDF
	\begin{eqnarray}
	f(x)=c x^{c-1} e^{-x^{c}},~~x>0,~c>1,
	\end{eqnarray}
	which is clearly log-concave with respect to $x$ since $\frac{d^{2}\log f(x)}{dx^2}=-\frac{c-1}{x^2}-c(c-1)x^{c-2}<0$, for $c>1.$ For this distribution, we have computed the WVE for $c=2$, which is given by ${\it VE}^w(X)\approx 1.87$, establishing that the WVE is not smaller than $1$ for a log-concave PDF in general.
\end{example}

\section{Weighted residual varentropy}
In this section, we study WVE for the residual lifetime of a system and discuss some properties. The residual lifetime is an important reliability metric of a system, which provides the information related to its remaining lifetime given that it has survived till time $t>0.$ It is denoted by $X_{t}=[X-t|X>t].$ The PDF of $X_{t}$ is given by $f_{t}(x)=\frac{f(x)}{\tilde{F}(t)},~x>t$, where $\tilde{F}(t)=P(X>t).$  Then, the weighted residual varentropy (WRVE) of $X$ is defined as
\begin{eqnarray}\label{eq3.1}
{\it VE}^w(X;t)=Var[IC^w(X_t)]
= \int_{t}^{\infty}\frac{f(x)}{\tilde F(t)}\bigg(w(x)\log\frac{f(x)}{\tilde F(t)}\bigg)^2dx-[H^w(X;t)]^2,
\end{eqnarray}
where  $IC^{w}(X_{t})=-w(X)\log f_{t}(X)$ and $H^{w}(X;t)$ is the weighted residual Shannon entropy (WRSE) with weight $w(x)$, given by
 \begin{align*}
H^{w}(X;t)=\int_{t}^{\infty}w(x)\frac{f(x)}{\tilde F(t)}\log \left(\frac{f(x)}{\tilde F(t)}\right)dx.
\end{align*}
The WRVE in (\ref{eq3.1}) can be further expressed for $w(x)=x$ as
\begin{eqnarray}\label{eq3.2}
{\it VE}^w(X;t)&=& \frac{1}{\tilde F(t)}\int_{t}^{\infty}x^2f(x)[\log f(x)]^2dx-\Lambda^2(t)E[X^2|X>t]-2\Lambda(t)H^{w^*}(X;t)\nonumber\\
&~&-[H^w(X;t)]^2,
\end{eqnarray}
where $\Lambda(t)=-\log \tilde F(t)$ is the cumulative hazard rate function, and $H^{w^*}(X;t)$ is the WRSE with weight  $w^*(x)=x^2$. It can be easily obtained that 
$$\lim\limits_{t\rightarrow0}{\it VE}^w(X;t)={\it VE}^w(X),$$
where ${\it VE}^w(X)$ is the WVE given in (\ref{eq2.1}). Next, we obtain closed-form expressions of the WRVE for some distributions.

\begin{figure}[h!]
	\begin{center}
		\subfigure[]{\label{c1}\includegraphics[height=1.25in]{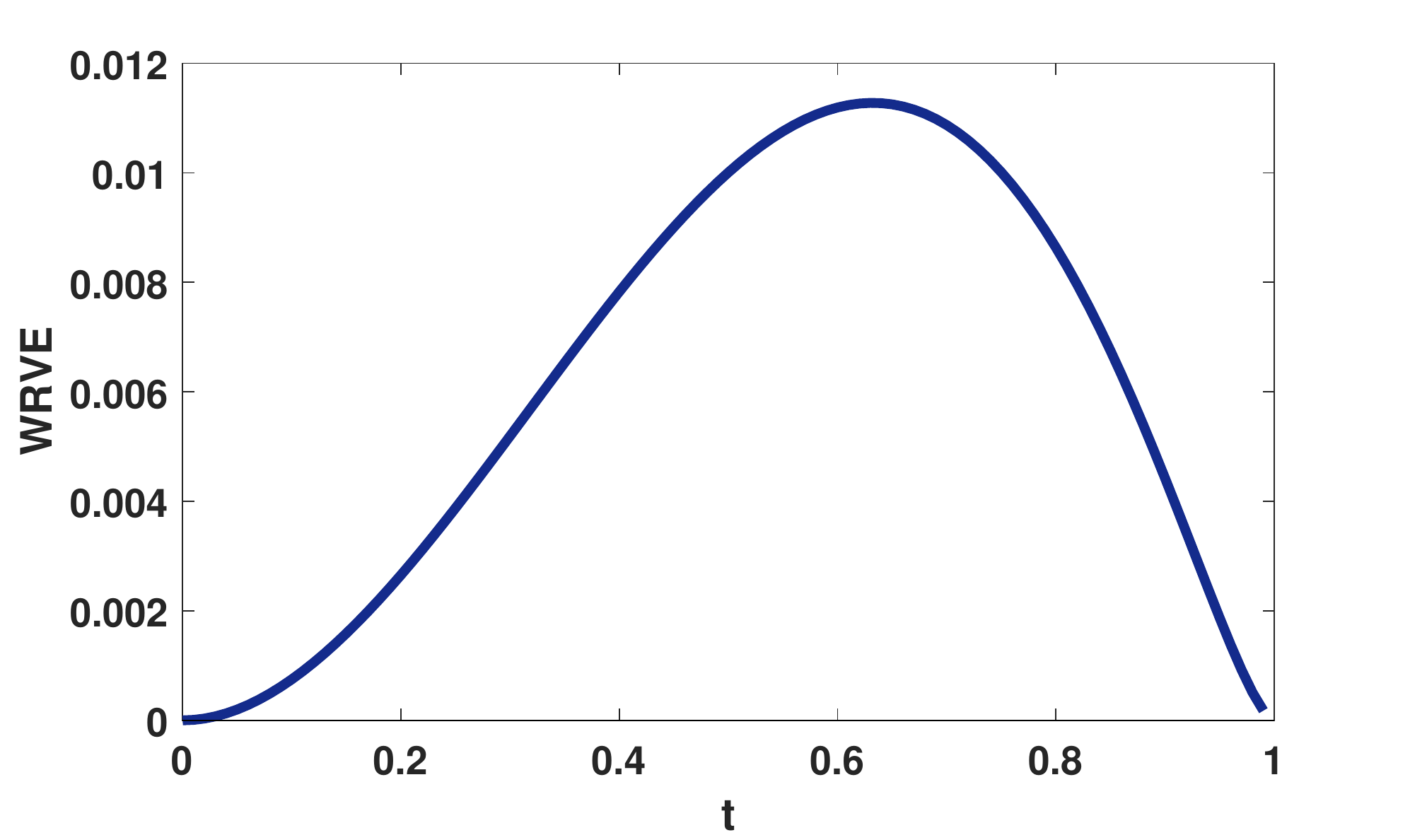}}
		\subfigure[]{\label{c1}\includegraphics[height=1.25in]{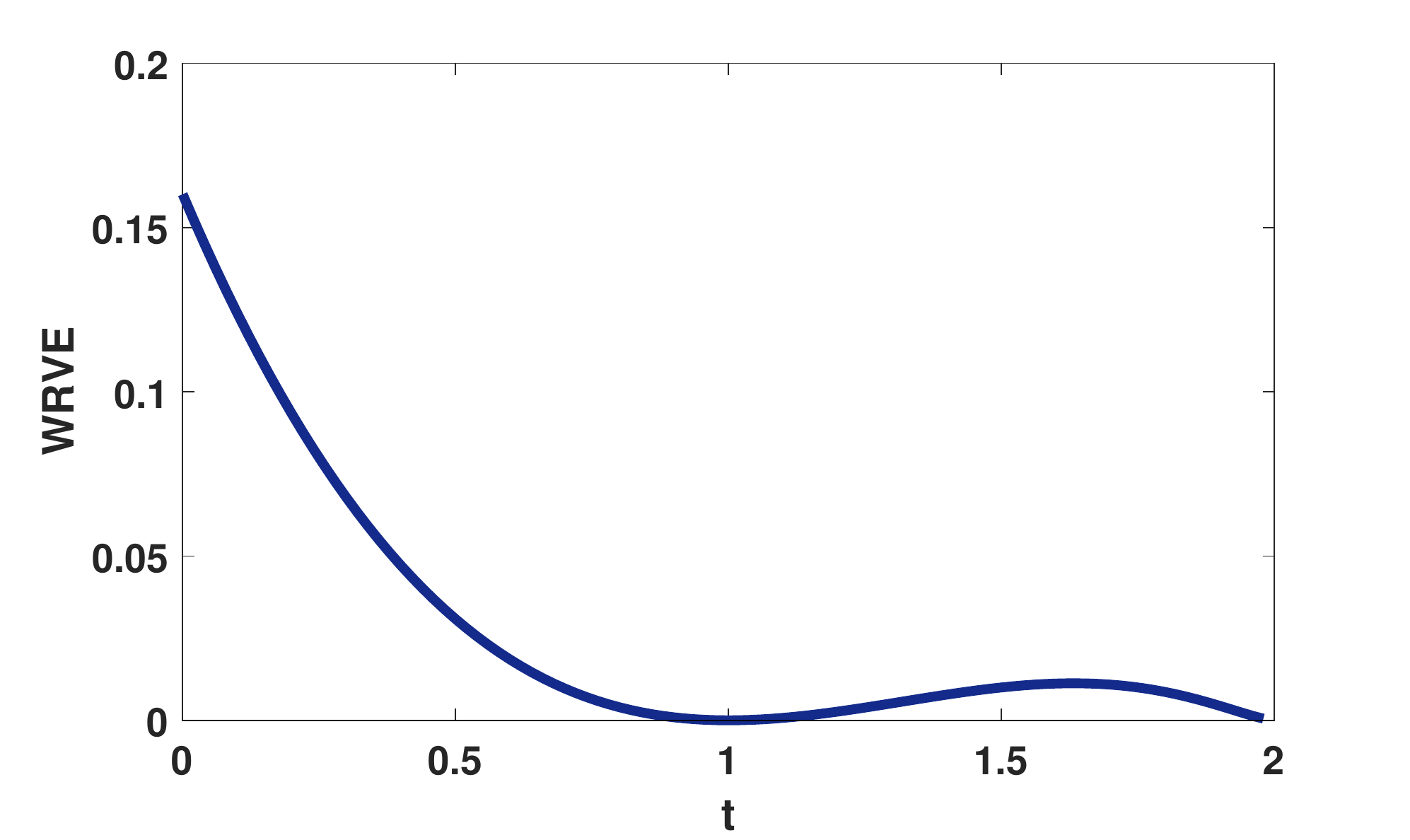}}
		\subfigure[]{\label{c1}\includegraphics[height=1.25in]{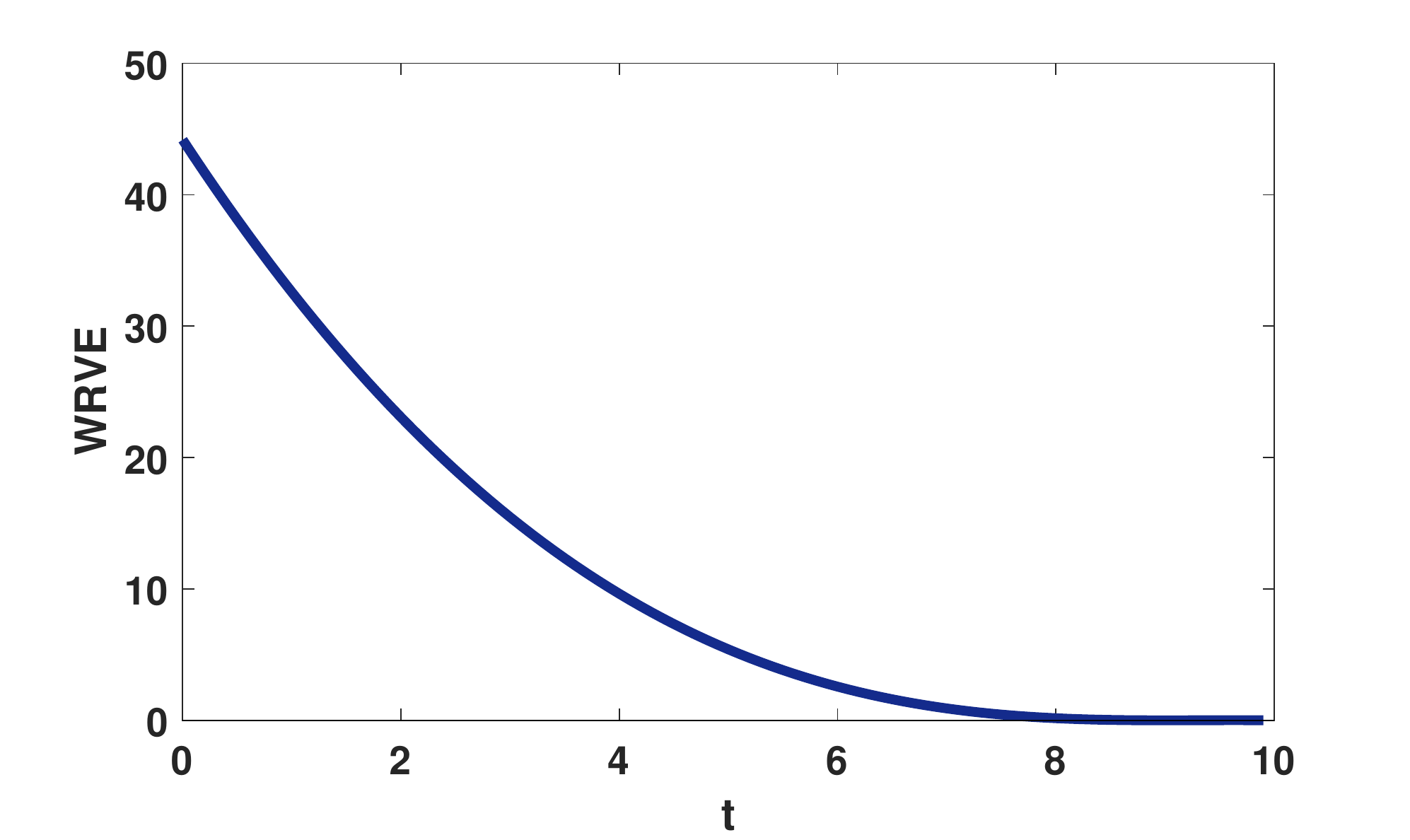}}
		\subfigure[]{\label{c1}\includegraphics[height=1.9in]{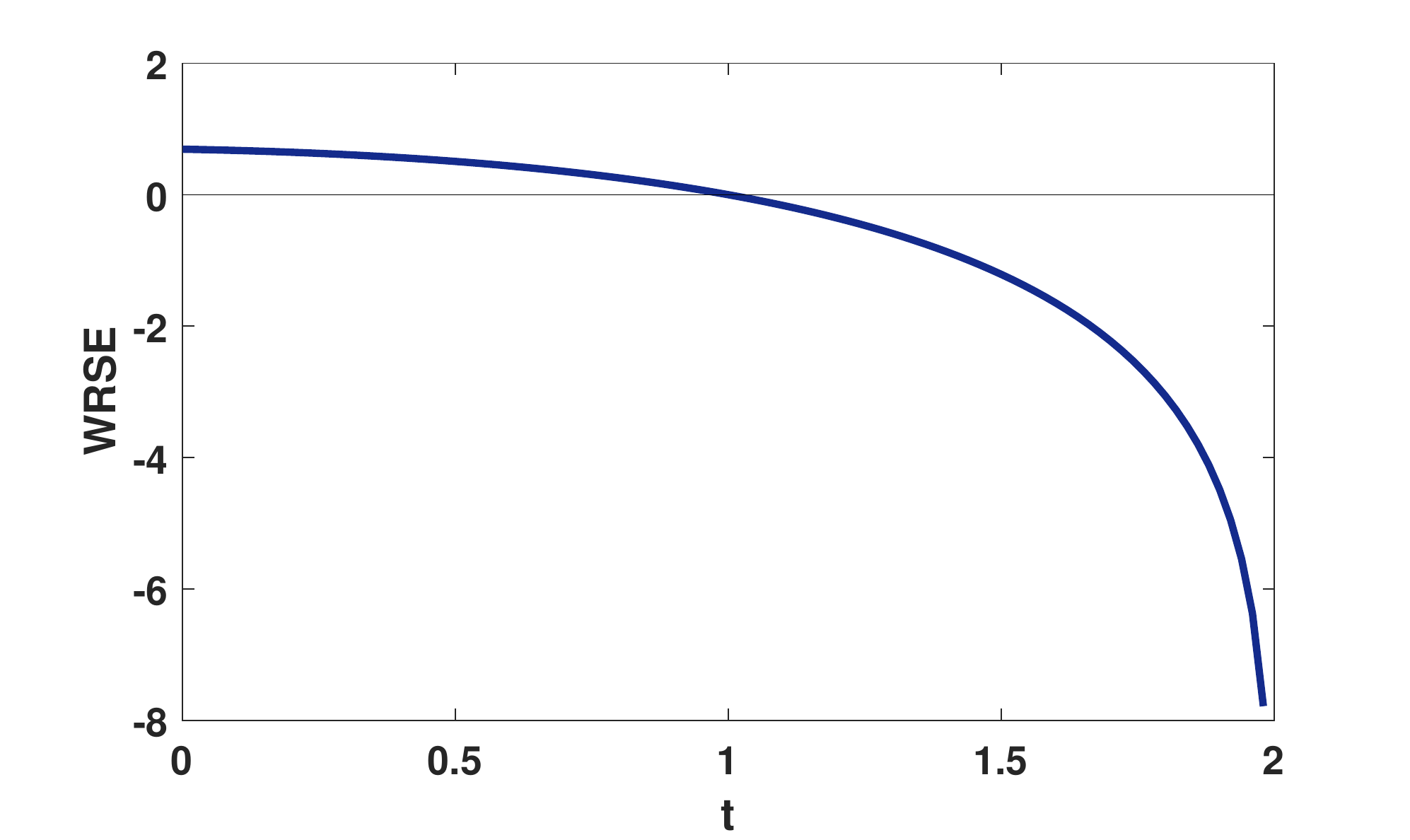}}
		\subfigure[]{\label{c1}\includegraphics[height=1.9in]{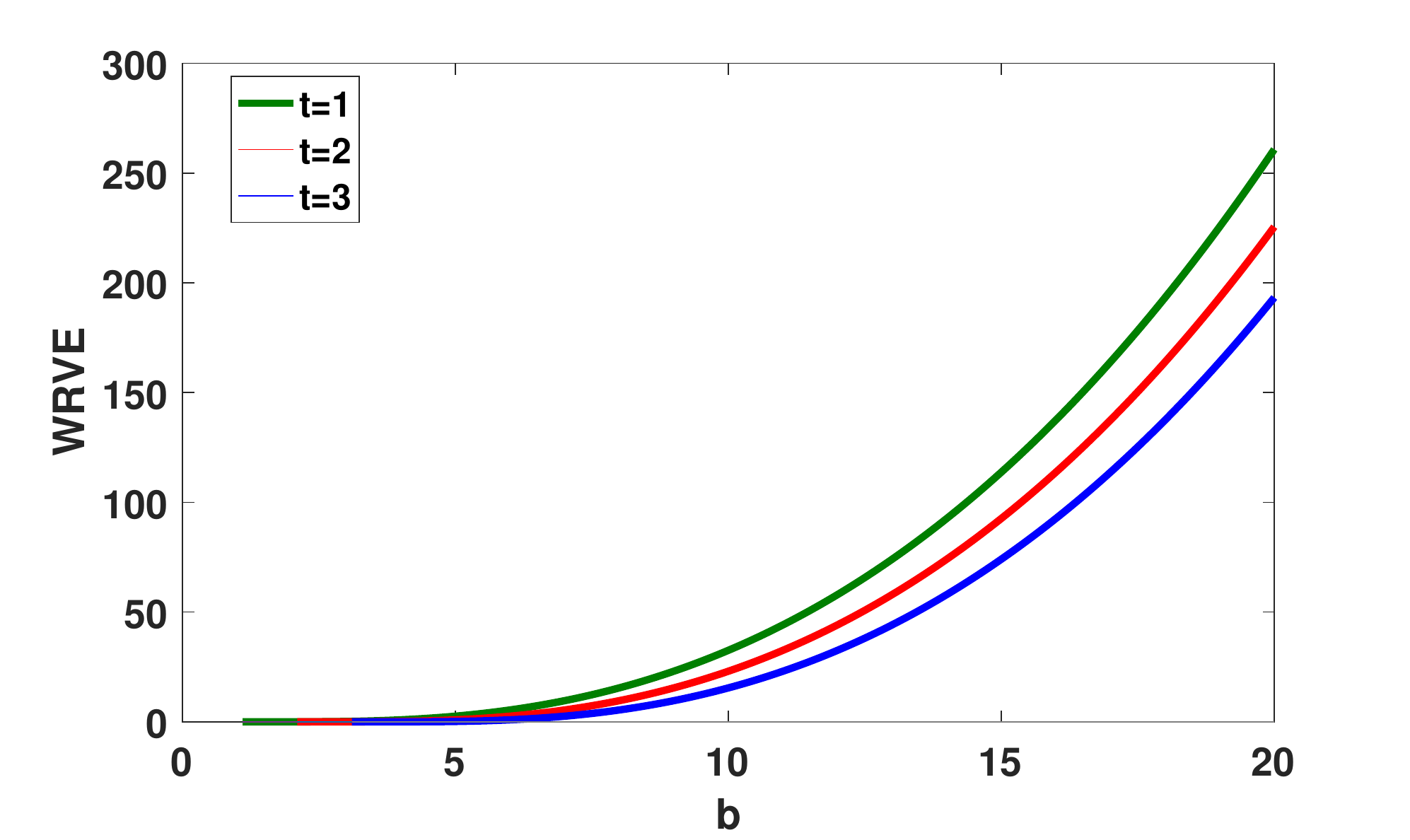}}
		\caption{Graphs of the WRVE with respect to $t$ for uniform distribution in the interval $(0,b)$ as in Example \ref{ex3.1}$(i)$ when $(a)$ $b=1$,  $(b)$  $b=2$, and $(c)$  $b=10$. $(d)$ Graph of the WRSE of uniform distribution with respect to $t$ in  $(0,2)$. $(e)$ Graphs of the WRVE of uniform distribution with respect to $b,$ for different values of $t.$}
	\end{center}
\end{figure}

 \begin{example}\label{ex3.1}~~
	\begin{itemize}
		\item[(i)] Suppose $X$ has uniform distribution in the interval $(0,b)$. Then, the WRVE is 
		$$ {\it VE}^x(X;t)=\frac{1}{12}[(b-t)\log(b-t)]^2, ~b>t>0.$$
		
		\item[(ii)] Let $X$ have exponential distribution with PDF $f(x)=\lambda e^{-\lambda x}, ~\lambda>0,~x>0$. Then, 
		\begin{eqnarray*}
			{\it VE}^x(X;t)&=&\frac{1}{\lambda^2}\big[(\lambda^2t^2+2\lambda t+2)(\log \lambda+\lambda t)^2-2(\lambda^3t^3+3\lambda^2t^2+6\lambda t+6)(\log \lambda+\lambda t)\\
			&~&+(\lambda^4t^4+4\lambda^3t^3+12\lambda^2t^2+24\lambda t+24)-\{(\lambda t+1)\log \lambda-(2+\lambda t)\}^2\big].
		\end{eqnarray*}

		\item[(iii)] Let $X$ be a power random variable with CDF $F(x)=x^k,~0<x<1,$ and $k>0$. Then,  
		\begin{eqnarray*}
			{\it VE}^x(X;t)&=&\frac{k}{(k+2)^3(1-t^k)}[2(k-1)^2-(k-1)^2t^{k+2}\{(k+2)^2(\log t)^2\\
			&~&-2(k+2)\log t\}+2(k-1)(k+2)\psi(k,t)\{t^{k+2}(1-(k+2)\log t)-1\}\\
			&~&+(k+2)^2\left(\psi(k,t)\right)^2(1-t^{k+2})]-\frac{k^2}{(k+1)^2(1-t^k)^2}[(k+1)\psi(k,t)\\
			&~&-t^{k+1}\{(k+1)\psi(k,t)+(k^2-1)\log t-(k-1)\}-(k-1)]^2,
		\end{eqnarray*}
	where $\psi(k,t)=\log(\frac{k}{1-t^k}).$
	\end{itemize}
\end{example}

We present the plots of the WRVE of uniform distribution with respect to $t$ as in Example \ref{ex3.1}$(i)$ in Figures $4$$(a)$, $(b),$ and $(c)$ to observe its  behaviour for different values of $b,$ say, $1$, $2$, and $10,$ respectively. From the figures, it is clear that WRVE of a distribution is not monotone in general. We have also presented graph of the WRSE of the uniform distribution with respect to $t$ in Figure $4(d)$ when $b=2.$ From Figures $4(b)$ and $(d)$, we see that the monotone behaviour of the WRVE and WRSE are not similar. In Figure $4(e)$, we have presented the graphs of the WRVE of the uniform distribution with respect to $b$, for some choices of $t$. Graphs of the WRVE and WRSE of the exponential distribution as in Example \ref{ex3.1}$(ii)$ are presented in Figures $5(a)$ and $(b)$, for different values of $\lambda.$  Figure $5(c)$ represents the plots of the WRVE of exponential distribution with respect to $\lambda,$ for $t=1,5,7.$ In Figures $6(a)$ and $6(b)$, we have plotted the graphs of the WRVE of power distribution with respect to $t$ and $k$, respectively, implying that the WRVE of power distribution is not monotone with respect to $t$ and $k$.

	\begin{figure}[h!]
	\begin{center}
		\subfigure[]{\label{c1}\includegraphics[height=1.26in]{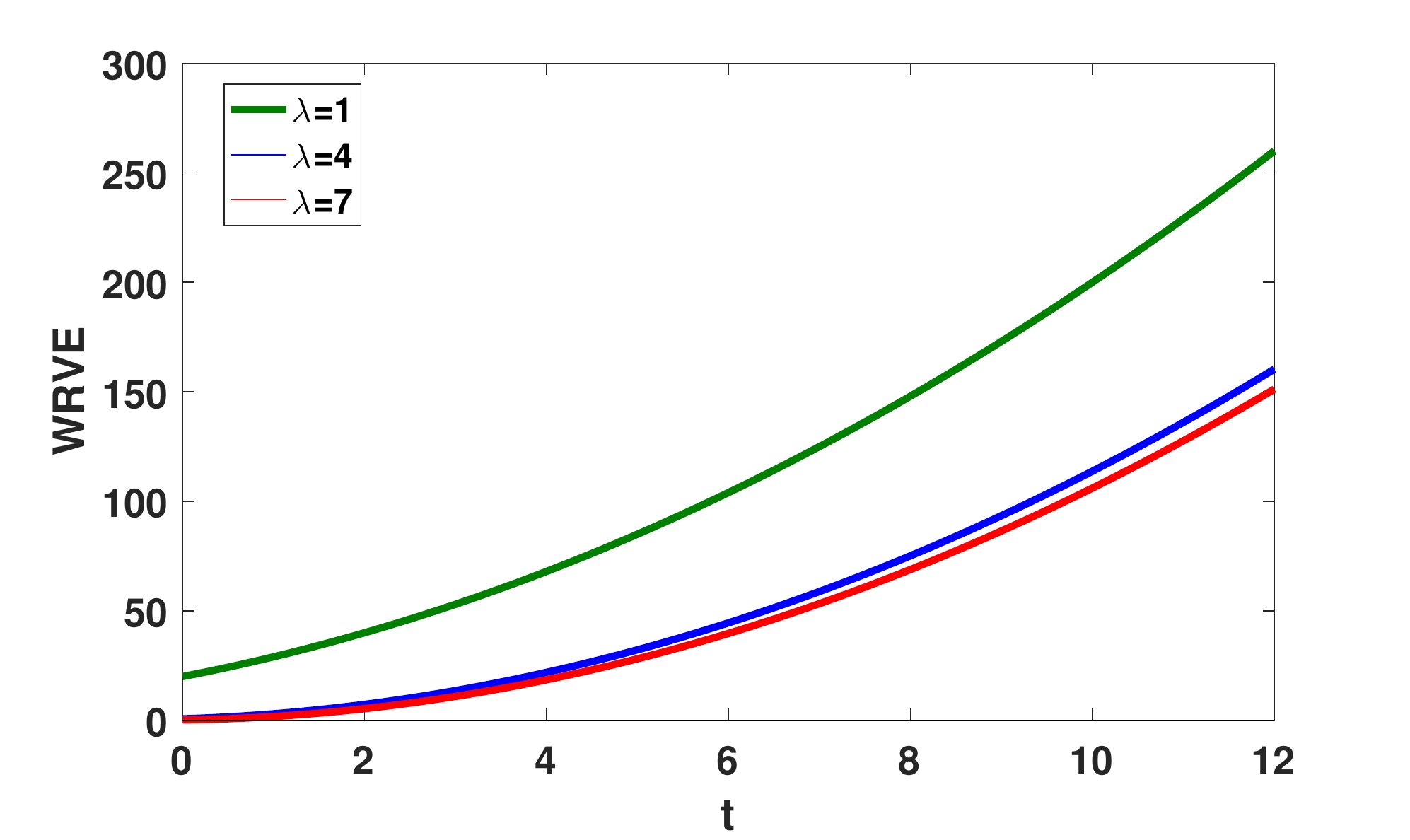}}
		\subfigure[]{\label{c1}\includegraphics[height=1.26in]{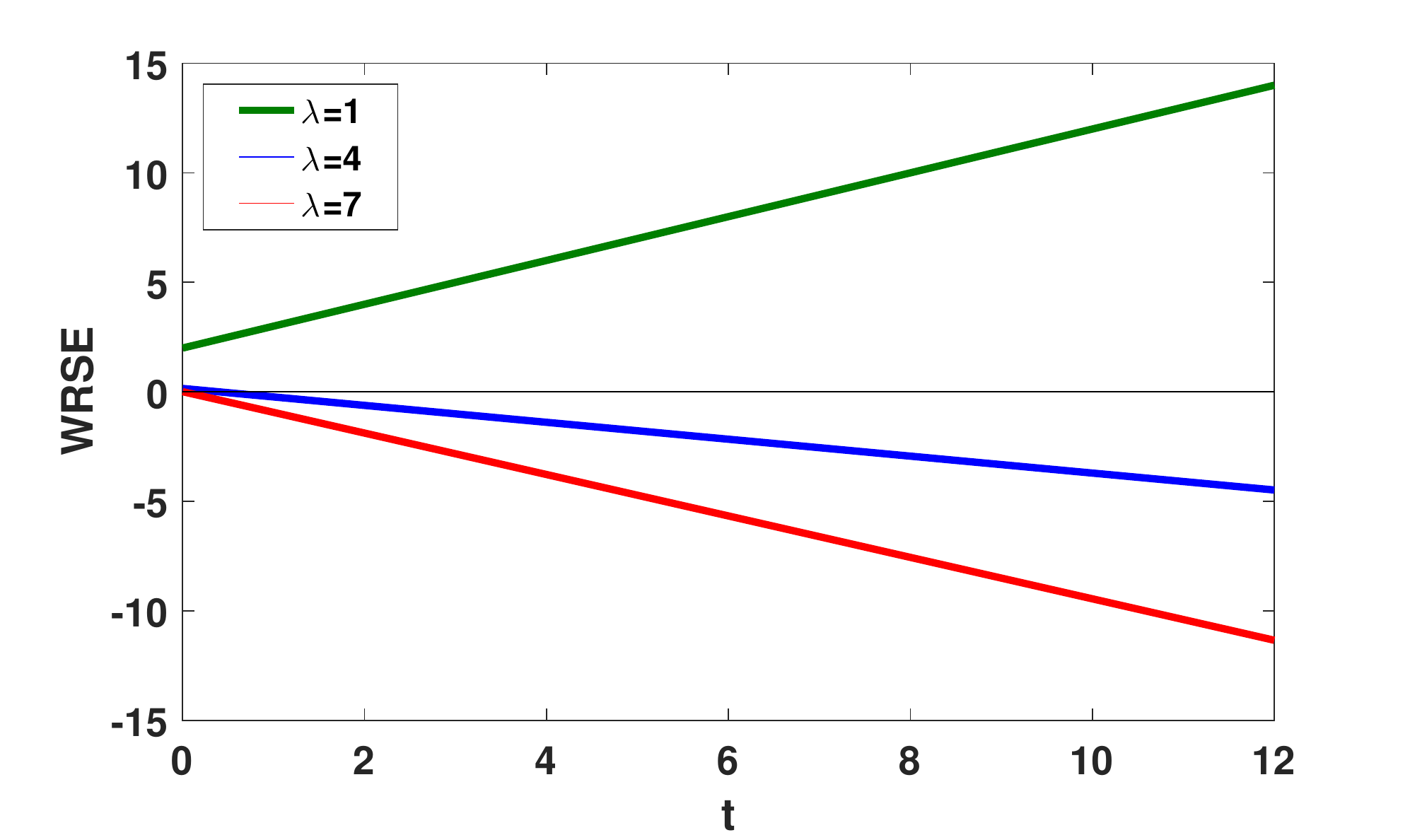}}
		\subfigure[]{\label{c1}\includegraphics[height=1.26in]{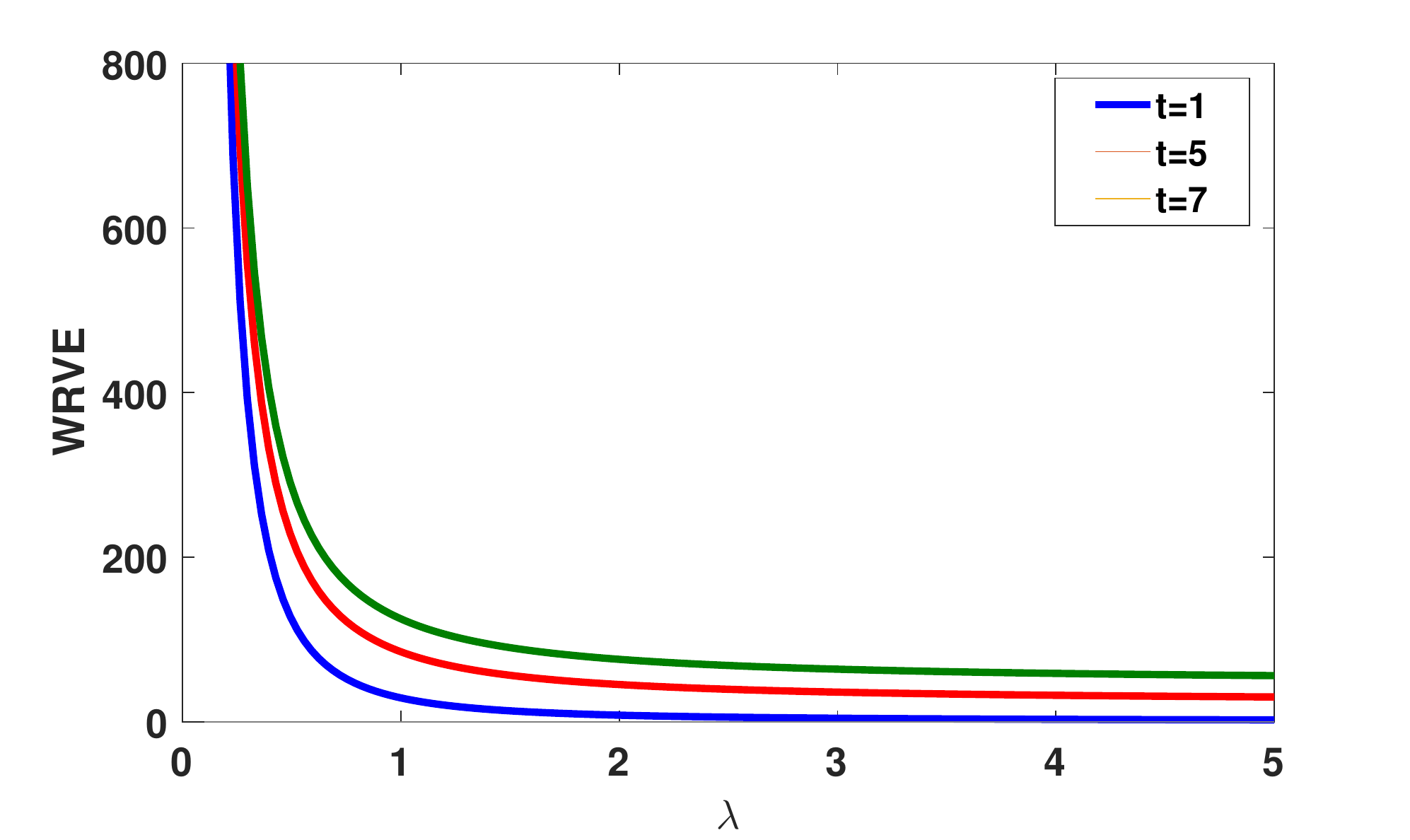}}
		\caption{Graphs of  $(a)$ WRVE and $(b)$ WRSE for exponential distribution with respect to $t$ as in Example \ref{ex3.1}$(ii)$ for $\lambda=1$, $\lambda=4,$ and $\lambda=7$. $(c)$ Graphs of the WRVE of exponential distribution with respect to $\lambda$, for different values of $t=1, 5, 7$ as in Example  \ref{ex3.1}$(ii)$. }
	\end{center}
\end{figure}

	\begin{figure}[h!]
	\begin{center}
		\subfigure[]{\label{c1}\includegraphics[height=1.9in]{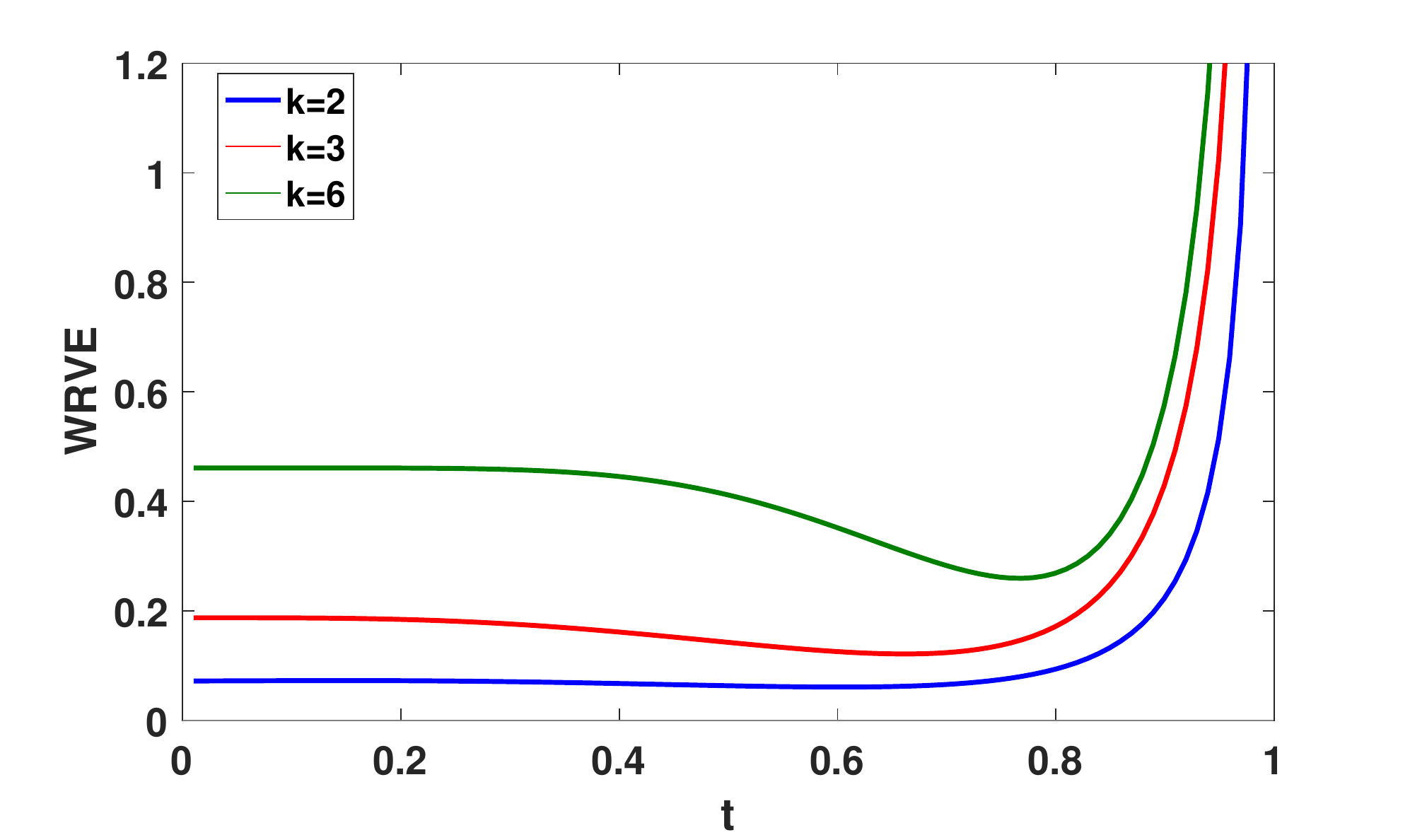}}
			\subfigure[]{\label{c1}\includegraphics[height=1.9in]{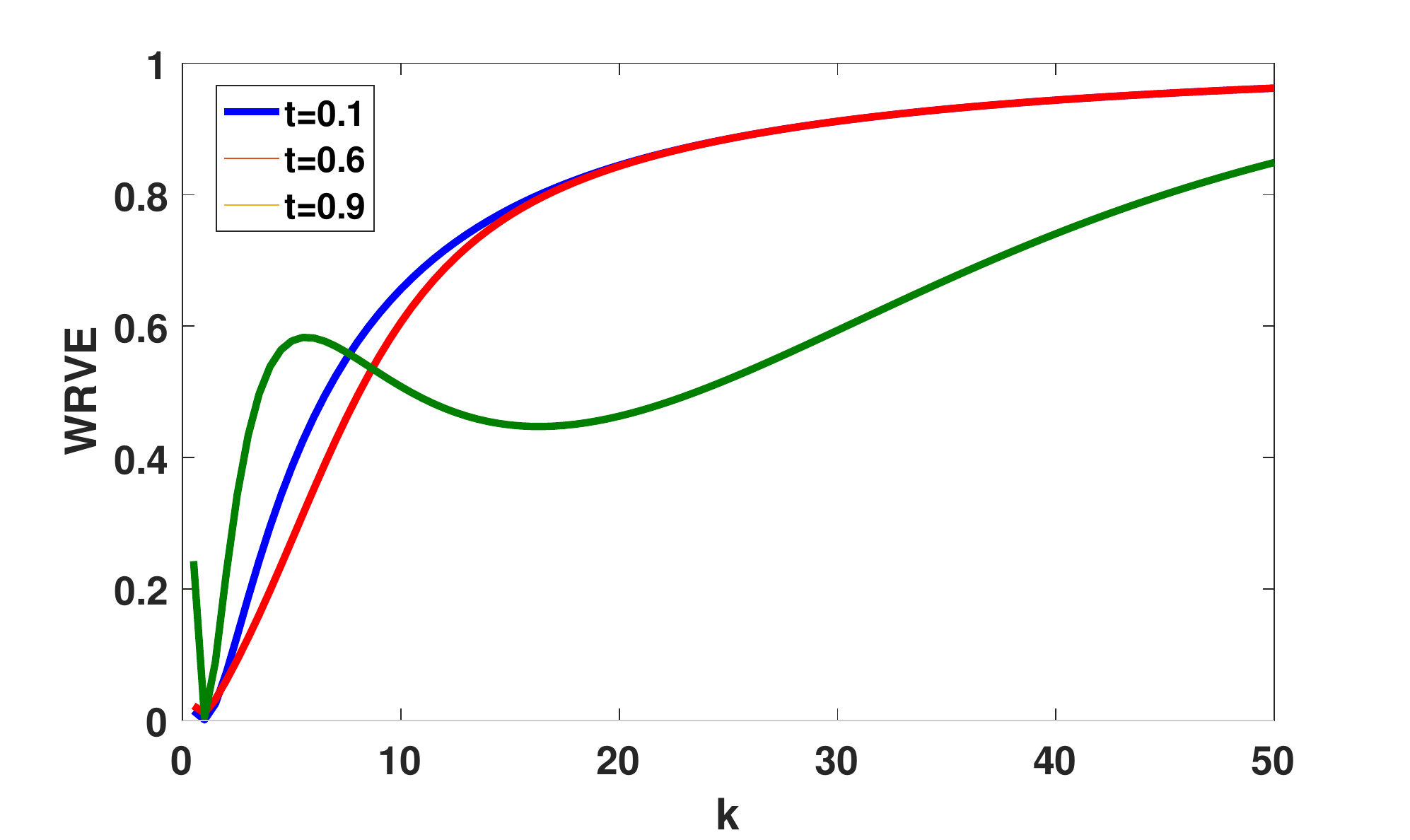}}
		\caption{Graphs of  the WRVE for power distribution as in Example \ref{ex3.1}$(iii)$ with respect to $(a)$ $t$ and $(b)$ $k$.}
	\end{center}
\end{figure}

The following result provides an explicit expression of  the derivative of WRVE with respect to $t$. We recall that the vitality function of an absolutely continuous random variable $X$ was proposed by \cite{kupka1989hazard}, which is defined as
\begin{eqnarray}
v(t)=E[X|X>t]=\int_{t}^{\infty}x\frac{f(x)}{\tilde F(t)}dx.
\end{eqnarray}
The vitality function is interpreted as the mean life span of a system whose age exceeds $t>0.$ The hazard rate is an important concept widely used in survival analysis and reliability theory. It provides the instantaneous failure rate of an item given that the item has survived till time $t.$ For a non-negative and absolutely continuous random variable $X$ the hazard rate is given by 
\begin{eqnarray}
r(t)=\frac{f(t)}{\tilde{F}(t)},~t>0,
\end{eqnarray}
such that $\tilde{F}(t)>0.$

 \begin{proposition}\label{pro3.1}
	We have
	\begin{eqnarray}\label{eq3.3}
	\frac{d}{dt}{\it VE}^w(X;t)&=&r(t)\{{\it VE}^w(X;t)-2(H^w(X;t)v(t)+H^{w^*}(X;t))^2\nonumber\\
	&~&-(H^w(X;t)+t\log r(t))^2\},
	\end{eqnarray}
	where $H^w(X;t)$ and $H^{w^*}(X;t)$ are the WRSEs with weights $w(x)=x$ and $w^*(x)=x^2$, respectively.
\end{proposition} 

  \begin{proof}
	Differentiating  (\ref{eq3.2}) with respect to $t$, we obtain
	\begin{eqnarray}\label{eq3.4}
	\frac{d}{dt}{\it VE}^w(X;t)&=&\frac{d}{dt}\bigg(\frac{1}{\tilde F(t)}\int_{t}^{\infty}x^2f(x)[\log f(x)]^2dx\bigg)-\frac{d}{dt}\big(\Lambda^2(t)E[X^2|X>t]\big)\nonumber\\
	&~&-2\frac{d}{dt}\big(\Lambda(t)H^{w^*}(X;t)\big)
	-\frac{d}{dt}\big([H^w(X;t)]^2\big)\nonumber\\
	&=& \frac{f(t)}{(\tilde {F}(t))^2}\int_{t}^{\infty}x^2f(x)[\log f(x)]^2dx-\frac{1}{\tilde {F}(t)}t^2f(t)[\log f(t)]^2\nonumber\\
	&~&-2\Lambda(t)r(t)E[X^2|X>t]-\Lambda^2(t)\frac{d}{dt}E[X^2|X>t]-2r(t)H^{w^*}(X;t)\nonumber\\
	&~&-2\Lambda(t)\frac{d}{dt}H^{w^*}(X;t)-2H^w(X;t)\frac{d}{dt}H^w(X;t).
	\end{eqnarray}
	Further, 
	\begin{eqnarray}\label{eq3.5}
	\frac{d}{dt}\big(H^{w^*}(X;t)\big)=r(t)\{H^{w^*}(X;t)+t^2\log r(t)-E[X^2|X>t]\},
	\end{eqnarray}
	\begin{eqnarray}\label{eq3.6}
	\frac{d}{dt}\big(E[X^2|X>t]\big)=r(t)\{E[X^2|X>t]-t^2\},
	\end{eqnarray}
	and
	\begin{eqnarray}\label{eq3.7}
	\frac{d}{dt}\big(H^{w}(X;t)\big)=r(t)\{H^{w}(X;t)+t\log r(t)-v(t)\}.
	\end{eqnarray}
	Now, using $(\ref{eq3.5})$,  $(\ref{eq3.6}),$  and  $(\ref{eq3.7})$ in  $(\ref{eq3.4})$, we get the required result. This completes the proof.
\end{proof}

%
%

Next, we obtain an upper bound of the WRVE. 

\begin{theorem}\label{th3.1}
		Let $X$ be a non-negative absolutely continuous random variable with PDF $f(\cdot)$ satisfying
		the condition in (\ref{eq2.10}).
		Then, for all $t>0$, we have
		\begin{eqnarray*}
			VE^x(X;t)&\leq&  H^{w_{2}}(X;t)+\Lambda^2(t)E[X^2|X>t],
		\end{eqnarray*}
		where  $H^{w_{2}}(X;t)$ is the WRSE with weight $w_{2}(x)=\alpha x^3+\beta x^2$.
	\end{theorem}
	
	\begin{proof}
		From  (\ref{eq3.1}) we obtain
		\begin{eqnarray}\label{eq3.11}
		VE^x(X;t)&=& \int_{t}^{\infty}x^2\frac{f(x)}{\tilde F(t)}[\log f(x)+\Lambda(t)]^2dx-[H^x(X;t)]^2\nonumber\\
		&\le&\int_{t}^{\infty}x^2\frac{f(x)}{\tilde F(t)}[\log f(x)+\Lambda(t)]^2dx.
		\end{eqnarray}
		Using (\ref{eq2.10}) in the right hand side integral  of (\ref{eq3.11}), we get
		\begin{eqnarray}\label{eq3.12}
		\int_{t}^{\infty}x^2\frac{f(x)}{\tilde F(t)}[\log f(x)+\Lambda(t)]^2dx&=& \int_{t}^{\infty}x^2\frac{f(x)}{\tilde F(t)}(\log f(x))^2dx+2\int_{t}^{\infty}x^2\frac{f(x)}{\tilde F(t)}\Lambda(t)\log f(x)dx\nonumber\\
		&~&+\frac{\Lambda^2(t)}{\tilde F(t)} \int_{t}^{\infty}x^2f(x)dx\nonumber\\
		&\leq& -\int_{t}^{\infty}x^2(\alpha x+\beta)\frac{f(x)}{\tilde F(t)}\log f(x)dx+\frac{\Lambda(t)}{\tilde F(t)}\int_{t}^{\infty}x^2f(x)dx\nonumber\\
		&~&-2\Lambda(t)\int_{t}^{\infty}x^2(\alpha x+\beta)\frac{f(x)}{\tilde F(t)}dx\nonumber\\
		&=&H^{w_{2}}(X;t)+\Lambda^2(t)E[X^2|X>t].
		\end{eqnarray}
		Now, using (\ref{eq3.12}) in (\ref{eq3.11}), the result readily follows. 
	\end{proof}
	
	Next, we consider an example for the purpose of illustration of the preceding theorem.
	
	\begin{example}\label{ex3.2}
		Let $X$ be exponentially distributed random variable with CDF $F(x)=1-e^{- x},~x\in(0,\infty)$. Then, it can be easily established that the condition $e^{-(\alpha x+\beta)}\leq f(x)\leq1,~x>0$,
		with $\alpha=1$ and $\beta=2$ is satisfied. Further,
		\begin{align*}
		H^{w_{2}}(X;t)+\Lambda^2(t)E[X^2|X>t]&=[\alpha(t^4+4t^3+12t^2+24 t+24)-(t^3+3t^2\\
		&+6 t+6)(-\beta+\alpha t)-(t^2+2 t+2)\beta t\\
		&+t^2(t^2+2 t+2)],
		\end{align*}
		where $\alpha=1$ and $\beta=2$. The closed form expression of ${\it VE}^x(X;t)$ is given in Example \ref{ex3.1}$(ii)$. Now, in Figure $7,$ the graphs of $H^{w_{2}}(X;t)+\Lambda^2(t)E[X^2|X>t]$ and  ${\it VE}^x(X;t)$ are plotted, illustrating the result in Theorem \ref{th3.1}.
\end{example}

\begin{figure}[htbp!]\label{fig3}
	\centering
	\includegraphics[width=14cm,height=8cm]{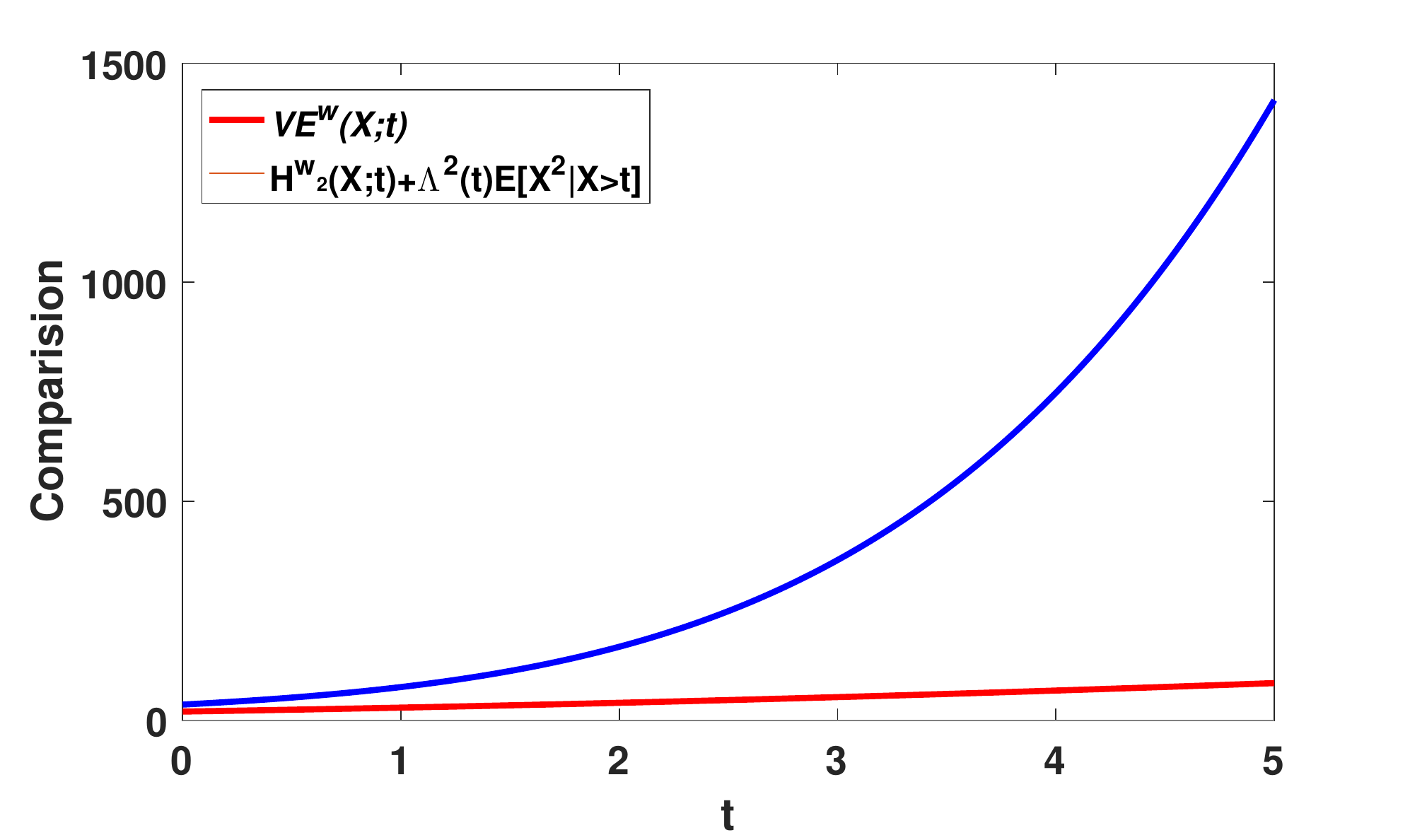}
	\caption{Graphs of $ H^{w_{2}}(X;t)+\Lambda^2(t)E[X^2|X>t]$ and $\it{VE}^x(X;t)$ as in Example \ref{ex3.2}.}
	\label{fig:PdfofIGLFRD*}
\end{figure}

Next, we obtain a lower bound of the WRVE in terms of the variance residual life (VRL), which is defined as 
\begin{eqnarray}
\sigma^{2}(t)=Var[X-t|X>t]=\frac{2}{\tilde{F}(t)}\int_{t}^{\infty}dv\int_{v}^{\infty}\tilde{F}(u)du-[\mu(t)]^{2},
\end{eqnarray}
where $\mu(t)$ is known as the mean residual lifetime (MRL), defined as
\begin{eqnarray}
\mu(t)=E[X-t|X>t]=\int_{t}^{\infty}\frac{\tilde{F}(x)}{\tilde{F}(t)}dx.
\end{eqnarray}

\begin{theorem}\label{th3.3}
	Suppose $X_t$ is the residual lifetime  with finite MRL $\mu(t)$ and finite VRL $\sigma^2(t)$. Then,
	\begin{eqnarray}
{\it VE}^x(X;t)\geq \sigma^2(t)\{1+E[-\eta_t(X_t)\log f_t(X_t)]+E[X_t\eta^{'}_t(X_t)]\}^2,
	\end{eqnarray}
	where $\eta_t(x)$ is a real-valued function, obtained from the relationship
	$$\sigma^2(t)\eta_t(x)f_t(x)=\int_{0}^{x}(\mu(t)-u)f_t(u)du, ~x>0.$$
\end{theorem}

\begin{proof}
	For an absolutely continuous random variable $X$ with PDF $f(\cdot)$, mean $\mu,$ and variance $\sigma^2$, the following inequality can be established (see \cite{cacoullos1989characterizations})
	\begin{eqnarray}
	Var[h(X)]\geq \sigma^2(E[\xi(X)h^{'}(X)])^2,
	\end{eqnarray}
	where a real-valued function $\xi(\cdot)$ is defined by the relationship $\sigma^2\xi(x)f(x)=\int_{0}^{x}(\mu(t)-u)f(u)du.$ In order to prove the result, we assume $X_t$ as a reference random variable, with $h(x)=IC^{x}(x)=-x\log f(x)$. Thus, 
	\begin{eqnarray}\label{eq3.18}
	VE^{x}(X;t)=Var[-X_t\log f_t(X_t)]&\geq& \sigma^2(t)\bigg\{E\big[\eta_t(X_t)\big(X_t\log f_t(X_t)\big)'\big]\bigg\}^2\nonumber\\
	&=&\sigma^2(t)\bigg\{E[-\eta_t(X_t)\log f_t(X_t)]-E\bigg[\eta_t(X_t)X_t\frac{f^{'}_t(X_t)}{f_t(X_t)}\bigg]\bigg\}^2.\nonumber\\
	\end{eqnarray}
	Further,
	\begin{eqnarray}\label{eq3.19}
	E\bigg[\eta_t(X_t)X_t\frac{f^{'}_t(X_t)}{f_t(X_t)}\bigg]&=& E\bigg[X_t\bigg(\frac{\mu(t)-X_t}{\sigma^2(t)}-\eta^{'}_t(X_t)\bigg)\bigg]\nonumber\\
	&=& -1-E[X_t\eta^{'}_t(X_t)].
	\end{eqnarray}
	Using (\ref{eq3.19}) in (\ref{eq3.18}), the required result readily follows, completing the proof of the theorem. 
\end{proof}

As a consequence of Theorem \ref{th3.3}, we get the following remark.

\begin{remark}~~
	\begin{itemize}
		\item[(i)] Let $\eta_t(x)$ be increasing in $x$. Then, 
		$${\it VE}^x(X;t)\geq\sigma^2(t)(E[\eta_t(X_t)\log f_t(X_t)])^2,~t\ge 0.$$
		\item[(ii)] For $f_t(x)\leq1,$ we have 
		$${\it VE}^x(X;t)\geq\sigma^2(t)(E[X_t\eta^{'}_t(X_t)])^2,~t\ge0.$$
	\end{itemize}
\end{remark}

Next, we study WRVE under general monotonic transformations.

\begin{theorem}\label{th3.4}
	Let $\phi(\cdot)$ be a strictly monotonic, continuous and differentiable function, and $Y=\phi(X)$, where $X$ is a non-negative absolutely continuous random variable. Then,
	\begin{equation}\label{eq3.20}
		{\it VE}^y(Y;t)=\left\{
		\begin{array}{ll}
		{\it VE}^{\phi}(X;\phi^{-1}(t)) -2H^{\phi}(X;\phi^{-1}(t))E[\eta_1(X)|X>\phi^{-1}(t)]\\
					+Var[\eta_1(X)|X>\phi^{-1}(t)]-2E[\phi(X)\eta_1(X) \log \frac{f(X)}{\tilde{F}(\phi^{-1}(t)}|X>\phi^{-1}(t)], \\~		
			if~\phi(\cdot)~ is ~strictly~ increasing,
			\\
			\\
		 \overline{{\it VE}}^{\phi}(X;\phi^{-1}(t)) -2\bar {H}^{\phi}(X;\phi^{-1}(t))E[\eta_2(X)|X\leq\phi^{-1}(t)]\\
				   			+Var[\eta_2(X)|X\leq\phi^{-1}(t)]-2E[\phi(X)\eta_2(X)\log \frac{f(X)}{F(\phi^{-1}(t))}|X\leq\phi^{-1}(t)],\\~	
			if~ \phi(\cdot)~ is ~strictly~ decreasing,
		\end{array}
		\right.
	\end{equation}
	where $\eta_1(x)=\phi(x)\log\phi'(x)$,  $\eta_2(x)=\phi(x)\log(-\phi'(x))$,  $H^{\phi}(X;t)=\int_{t}^{\infty}\phi(x)\frac{f(x)}{\tilde{F}(t)}\log \frac{f(x)}{\tilde{F}(t)}dx$, ${\it VE}^{\phi}(X;t)=\int_{t}^{\infty}\phi^2(x)\frac{f(x)}{\tilde{F}(t)}(\log \frac{f(x)}{\tilde{F}(t)})^2dx-[H^{\phi}(X;t)]^2$,
	$\bar {H}^{\phi}(X;t)=\int_{0}^{t}\phi(x)\frac{f(x)}{{F}(t)}\log \frac{f(x)}{{F}(t)}dx$, 
	and
	$\overline{{\it VE}}^{\phi}(X;t)=\int_{0}^{t}\phi^2(x)\frac{f(x)}{{F}(t)}(\log \frac{f(x)}{{F}(t)})^2dx-[\bar {H}^{\phi}(X;t)]^2$.
\end{theorem}

\begin{proof}
	Let $\phi(\cdot)$ be strictly increasing. Then, after some calculations, we get 
	\begin{eqnarray}\label{eq3.21}
		{\it VE}^y(Y;t)&=&\int_{\phi^{-1}(t)}^{\phi^{-1}(\infty)}\phi^2(x)\frac{f(x)}{\tilde{F}(\phi^{-1}(t))}\bigg[\log \bigg\{\frac{f(x)}{\tilde{F}(\phi^{-1}(t))} \bigg\}\bigg]^2dx\nonumber\\
		&~&-E[\phi^2(X)\{2\log \phi^{'}(X) \log \frac{f(X)}{\tilde{F}(\phi^{-1}(t)}-(\log \phi^{'}(X))^2\}|X>\phi^{-1}(t)]\nonumber\\
		&~&-\{H^{\phi}(X;\phi^{-1}(t))+E[\phi(X)\log \phi^{'}(X)|X>\phi^{-1}(t)]\}^2.
	\end{eqnarray}
	Now, the rest of the proof readily follows from (\ref{eq3.21}). The proof when $\phi(\cdot)$ is strictly decreasing is analogous.   	 
\end{proof}

Below, we provide an example for the illustration of Theorem \ref{th3.4}.
\begin{example}\label{ex3.3*}
	Let $X$ be a random variable with CDF $F(x)=1-e^{-\lambda x},~ x>0$ and $\lambda>0$. We consider the transformation $Y=\phi(X)=X^2$. Here, $\phi(\cdot)$ is strictly increasing. The random variable $Y$ follows Weibull distribution with CDF $G(x)=1-e^{-\lambda \sqrt{x}},~x>0,~\lambda>0.$  Now, using Theorem \ref{th3.4}, we obtain the WRVE of $Y$ as
	\begin{eqnarray*}
		{\it VE}^y(Y;t)&=&\frac{12}{\lambda}(\log \lambda+\lambda \sqrt{t}-5)^2+2\lambda t\sqrt{t}(3-\log \lambda-\lambda \sqrt{t})\left(t+\frac{5}{\lambda}\sqrt{t}+\frac{20}{\lambda^2}\right)\\
		&~&-\left\{(\log \lambda+\lambda \sqrt{t}-3)\left(t+\frac{2\sqrt{t}}{\lambda}+\frac{2}{\lambda^2}\right)-\lambda t\sqrt{t}\right\}^2+Var[X^2\log (2X)|X>\sqrt{t}]\\
		&~& -2\left\{(\log \lambda+\lambda \sqrt{t}-3)\left(t+\frac{2\sqrt{t}}{\lambda}+\frac{2}{\lambda^2}\right)-\lambda t\sqrt{t}\right\}E[X^2\log (2X)|X>\sqrt{t}]\\
		&~&+(\log \lambda+\lambda \sqrt{t})^{2}\left(t^2+\frac{4t\sqrt{t}}{\lambda}\right)+\frac{60}{\lambda^4}(1+\lambda \sqrt{t})^2+\left(\lambda^2 t^3+\frac{60}{\lambda^4}\right)\\
		&~&-2E\left[X^4\log (2X)\log \frac{\lambda e^{-\lambda X}}{e^{-\lambda\sqrt{t}}}|X>\sqrt{t}\right],
	\end{eqnarray*}
which has been depicted in Figure $8$ with respect to $\lambda$, for different values of $t$.
\end{example}

\begin{figure}[htbp!]
	\centering
	\includegraphics[width=14cm,height=8cm]{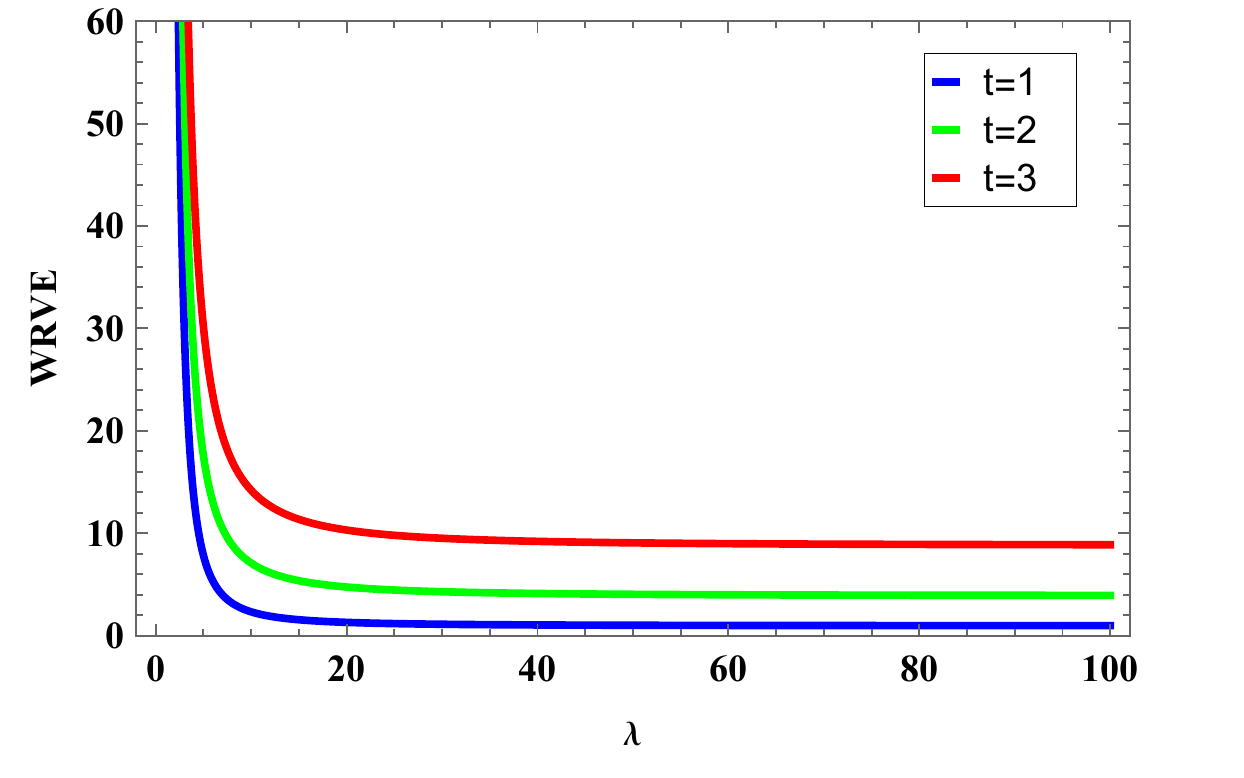}
	\caption{Graphs of the WRVE with respect to $\lambda$ for $t=1,2,5$  as in Example \ref{ex3.3*}}.
	\label{fig:PdfofIGLFRD*}
\end{figure}

Next, we study WRVE under affine transformations.
 \begin{corollary}\label{cor3.3}
	Suppose $X$ is a non-negative absolutely continuous random variable. Further, suppose $Y=aX+b,$ with $a>0$ and $b\geq0$. Then, 
	\begin{eqnarray*}
	{\it VE}^y(Y;t)&=& {\it VE}^{w_1}\left(X;\frac{t-b}{a}\right)
		+(\log a)^2E\bigg[aX+b|X>\frac{t-b}{a}\bigg]\bigg\{1-E\bigg[aX+b|X>\frac{t-b}{a}\bigg]\bigg\}\\
		&~&-2\log aH^{w_1}\left(X;\frac{t-b}{a}\right)\bigg\{1+E\bigg[aX+b|X>\frac{t-b}{a}\bigg]\bigg\},
	\end{eqnarray*}
	where ${\it VE}^{w_1}(X;\frac{t-b}{a})$ and $H^{w_1}(X;\frac{t-b}{a})$ are the WRVE and WRSE, respectively, with weight function $w_{1}(x)=ax+b$.
\end{corollary}
\begin{proof}
	The CDF and PDF of $Y$ are respectively given by $G(x)=F(\frac{x-b}{a})$ and $g(x)=\frac{1}{a}f(\frac{x-b}{a})$, $x>b$. Now, using these, the proof of the corollary follows after some simplification, and thus it is omitted.
\end{proof}

An application of Corollary \ref{cor3.3} is provided in the following example.
\begin{example}\label{ex3.3}
	Suppose $X$ is a random variable following Pareto-I distribution with CDF $F(x)=1-(\frac{\alpha}{x})^\beta,~x>\alpha>0,~\beta>0$. Consider a transformation $Y=X-b$, $b>0.$ Then, for $\alpha=1$ and $\beta=2$, we get from Corollary \ref{cor3.3} as
	\begin{eqnarray*}
	{\it VE}^y(Y;t)&=&\log (3(t+b)^3)\{(3t^2+3bt+b^2)\log (3(t+b))+\frac{8}{3}(9t^2+9bt+b^2)\}\\
		&~&+16(3t^2+3bt+b^2)(\log (t+b))^2+\frac{16}{3}(18t^2+27tb+11b^2)\log (t+b)\\
		&~&-\{2(3t+b)\log (t+b)+\frac{1}{2}(3t+b)\log(3(t+b)^3)+(3t+\frac{5b}{3})\}^2\\
		&~&+\frac{8}{9}(108t^2+189tb+83b^2)-8\log (3(t+b)^3) \log(t+b)(3t^2+3bt+b^2),
	\end{eqnarray*} 
which has been plotted in Figure $9$ with respect to $b.$ Here, we have considered three different choices of $t.$
\end{example}

\begin{figure}[htbp!]\label{fig3}
	\centering
	\includegraphics[width=14cm,height=8cm]{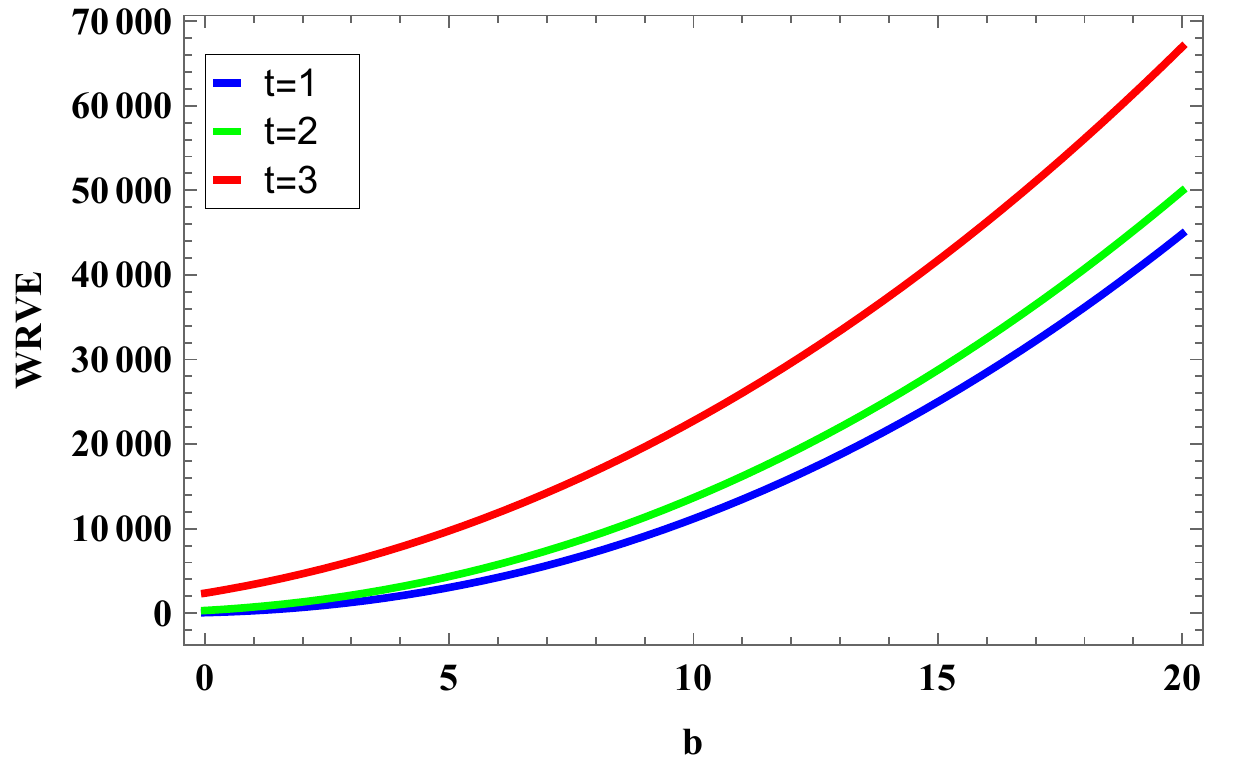}
	\caption{Graphs of the WRVE with respect to $b$ for $t=1,2,5$ as in Example \ref{ex3.3}}.
	\label{fig:PdfofIGLFRD*}
\end{figure}

It is important to see the role of various uncertainty measures in real-life problems. \cite{burnhan2002model} suggested that a distribution having more information should be preferred. Here, we present effectiveness of the proposed measure through the comparison of WRVE given in $(\ref{eq3.1})$ with  RVE for exponential and uniform distributions. We recall that the RVE was introduced by \cite{di2021analysis}. The authors have also explored various properties of this measure.
\begin{itemize}
	\item [(a)]
	In Figure $10,$ we plot the WRE of exponential distribution obtained in Example \ref{ex3.1}$(ii)$ and the RVE of exponential distribution, which  is equal to $1$. From Figure $10$, it is observed that the graph of WRVE covers larger area than that of the RVE, for different values of $\lambda=1,5,7,12$. 
	\begin{figure}[htbp!]\label{fig3}
		\centering
		\includegraphics[width=14cm,height=8cm]{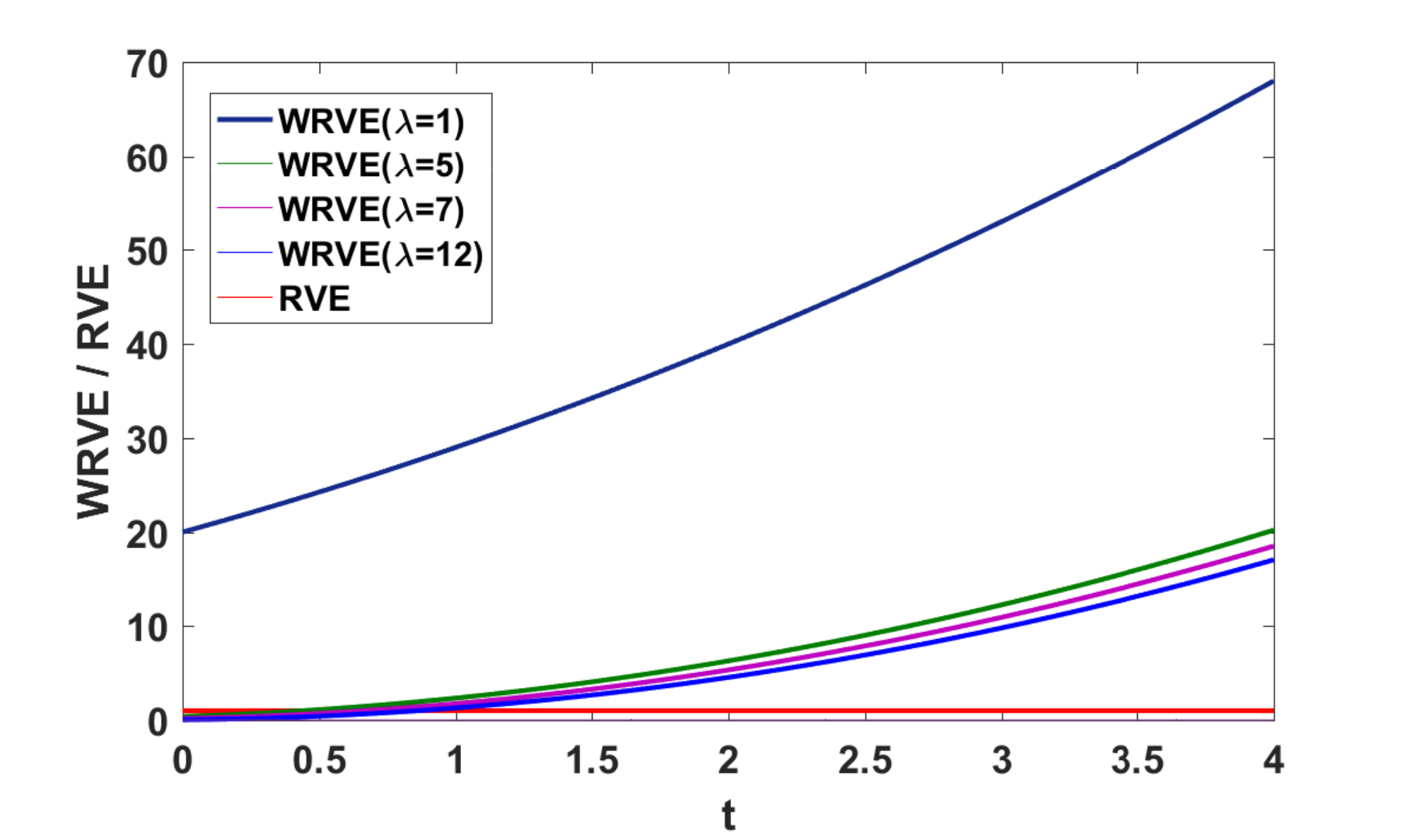}
		\caption{Graphs of the WRVE with respect to $t$ and RVE for different values of $\lambda$.}
		\label{fig:PdfofIGLFRD*}
	\end{figure}

	\item [(b)]
	The graphs of WRVE and RVE for uniform distribution in $(0,b)$ as in Example \ref{ex3.1}$(i)$ are plotted in Figure $11$ for different choices of $b$. Note that the RVE for uniform distribution in $(0,b)$ is zero. From Figures $11(a-d),$ we observe that the WRVE covers larger area than that of the RVE. 
	
	\begin{figure}[h!]
		\begin{center}
			\subfigure[]{\label{c1}\includegraphics[height=1.8in]{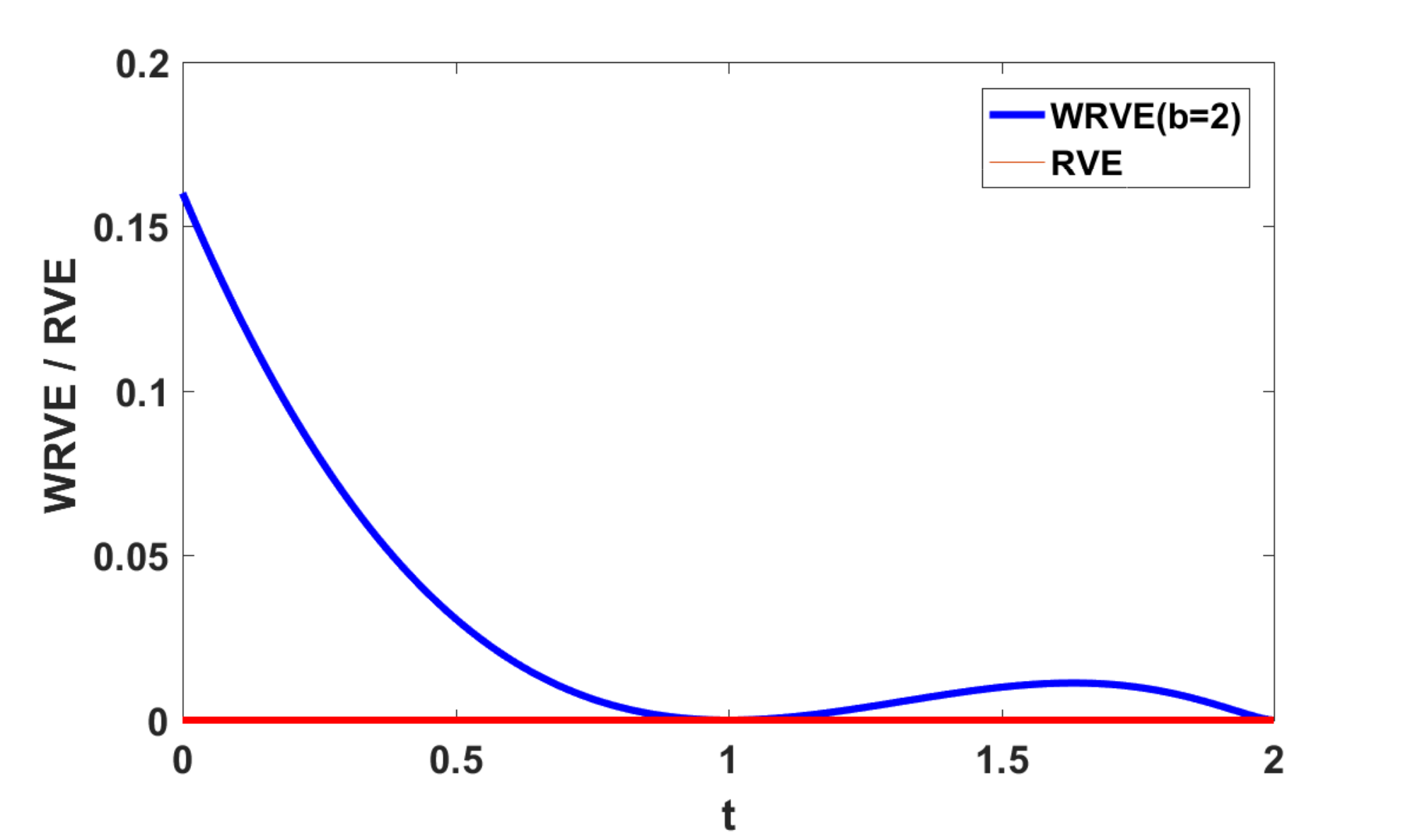}}
			\subfigure[]{\label{c1}\includegraphics[height=1.8in]{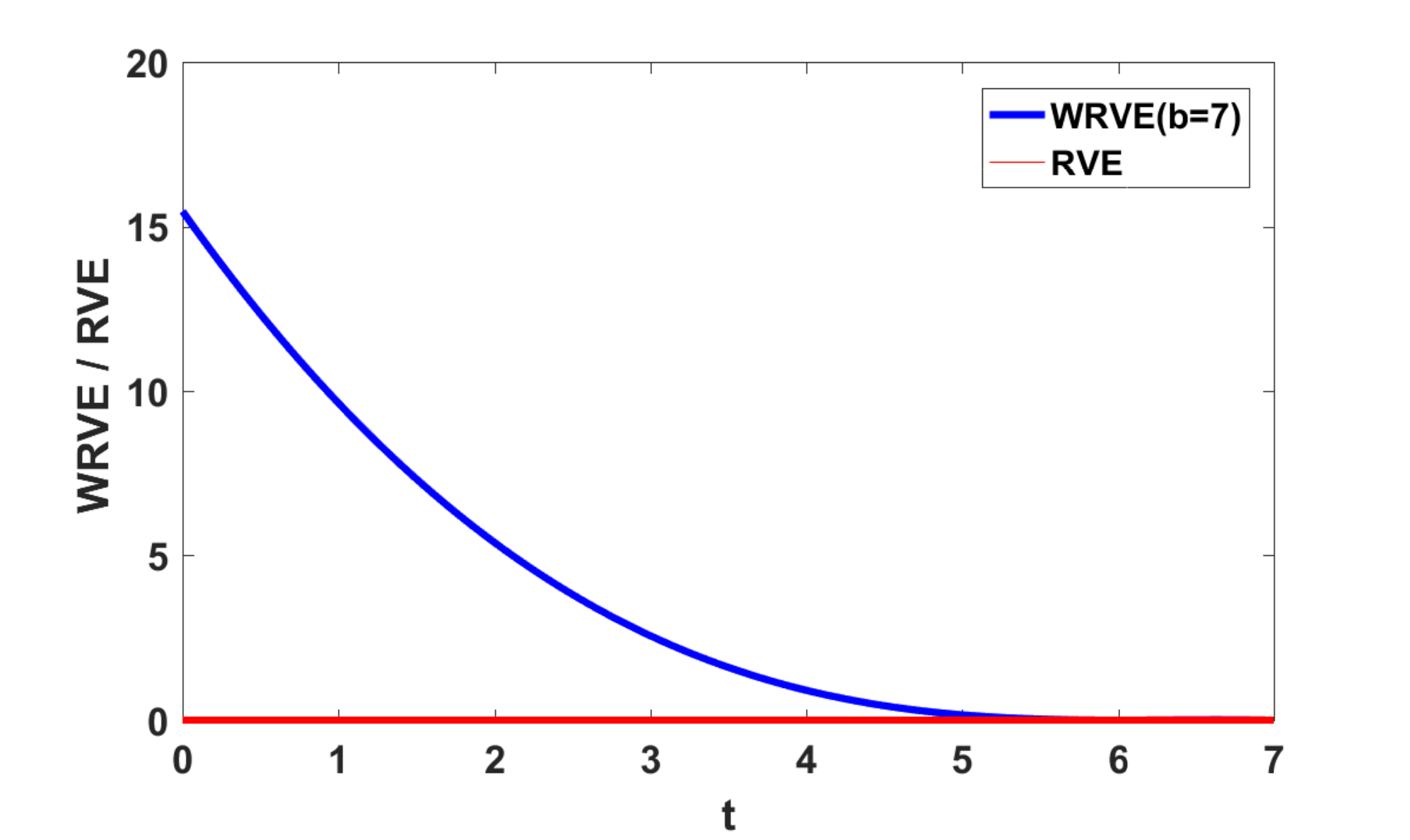}}
			\subfigure[]{\label{c1}\includegraphics[height=1.8in]{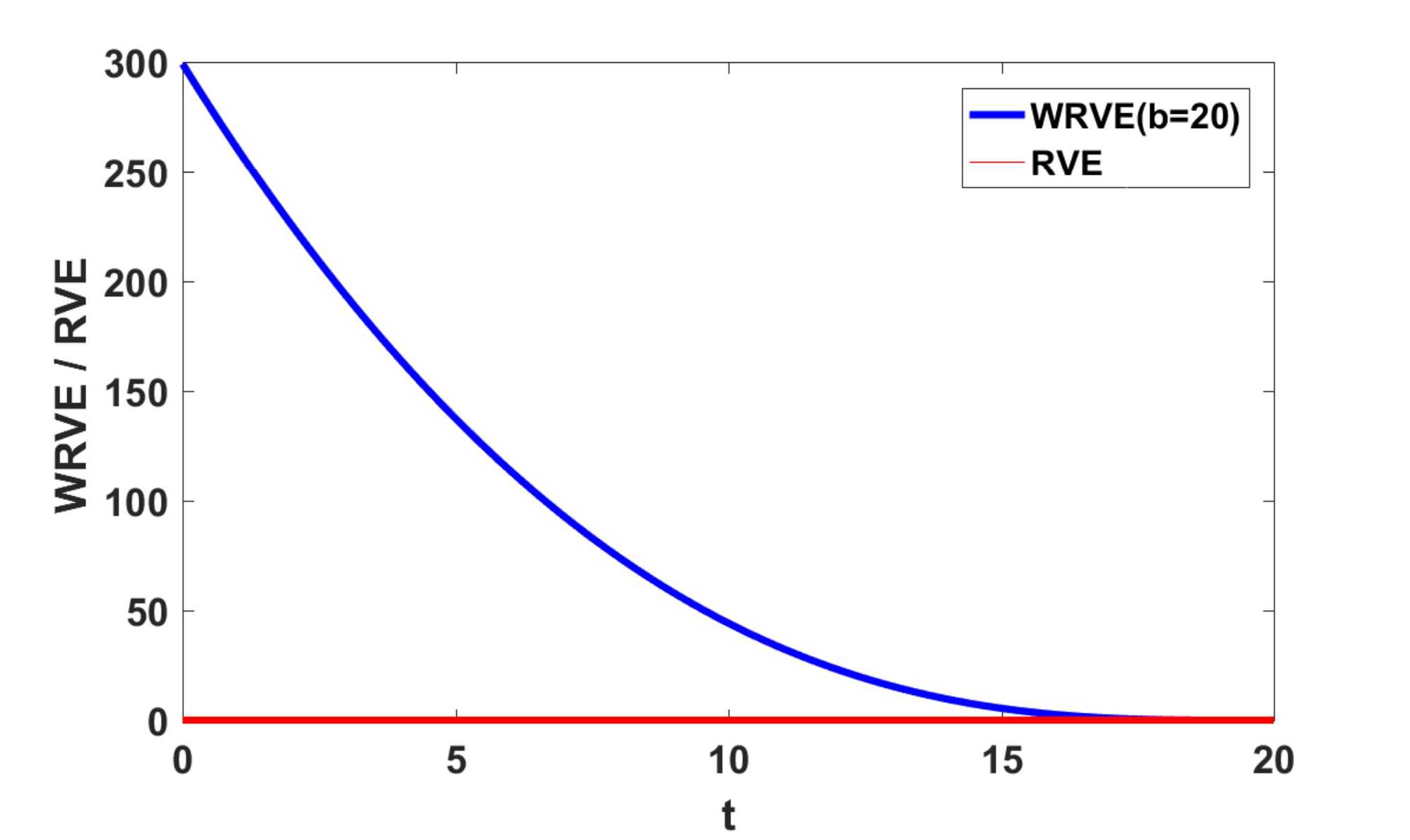}}
			\subfigure[]{\label{c1}\includegraphics[height=1.8in]{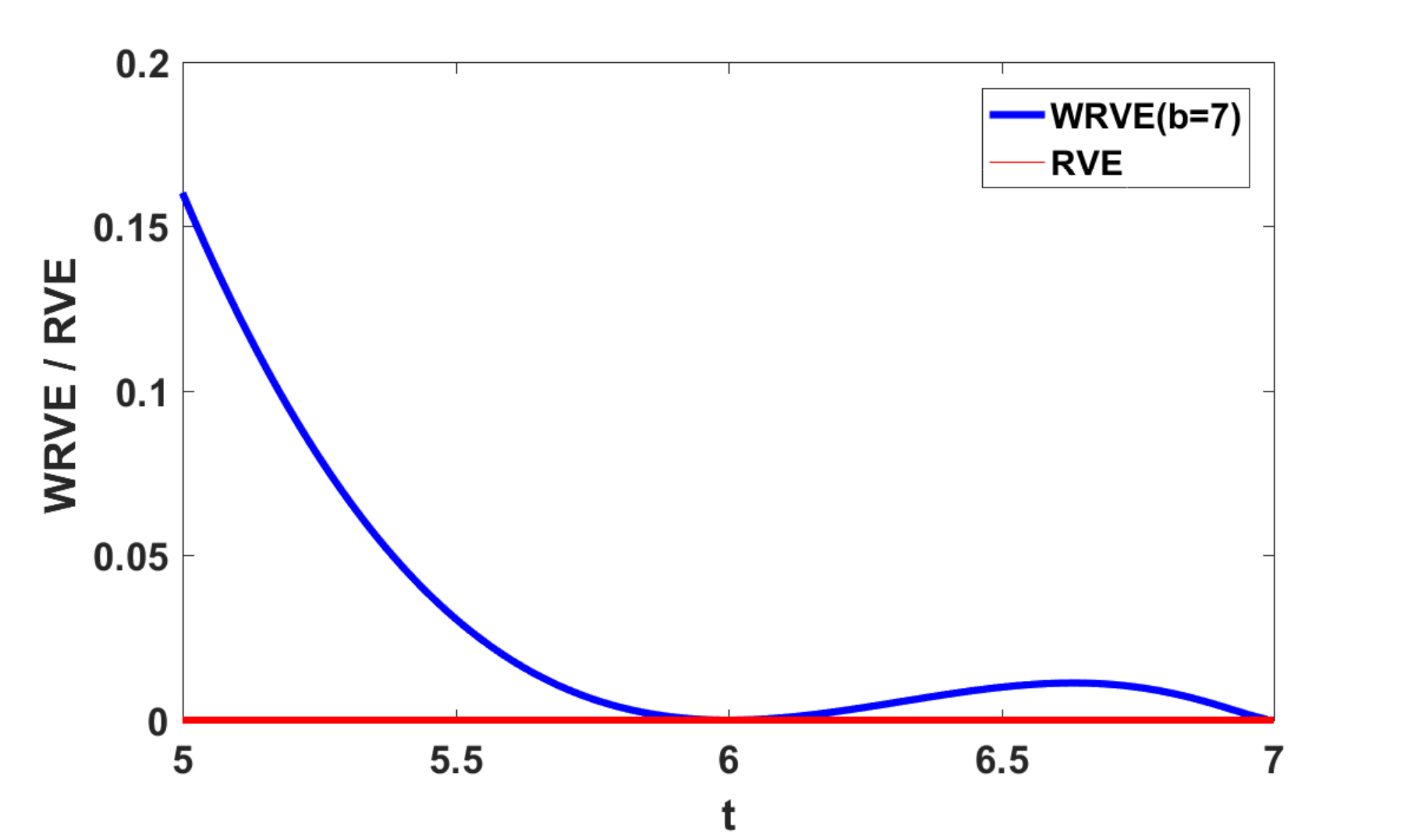}}
			\caption{Graphs of  the WRVE for the uniform distribution in the interval $(0,b)$ with respect to $t$ as in Example \ref{ex3.1}$(i)$ when $(a)$ $b=2$, $(b)$ $b=7$, and $(c)$ $b=20.$ $(d)$ Graph of WRVE for the uniform distribution in the interval $(0,7)$ with respect to $t$ as in Example \ref{ex3.1}$(i)$ when $t\in(5,7).$}
		\end{center}
	\end{figure}
\end{itemize}

\section{Weighted varentropy of coherent systems} 
In this section, we study WVE of coherent systems.  The coherent systems are useful in various real life applications such as industrial systems, telecommunications and oil pipeline systems. For details, please refer to \cite{elsayed2012}. Consider a coherent system with $n$ components. The lifetime of the coherent system is denoted by $T$. The random lifetimes of $n$ components of the coherent system are independent and identically distributed with a common CDF $F(\cdot)$. The CDF of $T$ is expressed  as
\begin{eqnarray}\label{eq4.1}
F_T(x)=q(F(x)),
\end{eqnarray}
where $q:[0,1]\rightarrow[0,1]$ is the distortion function (see \cite{navarro2013stochastic}). Various authors have studied well-established information measures for coherent systems. For instance, see the references by  \cite{toomaj2017some}, \cite{cali2020properties}, and \cite{saha2023extended}. We recall that the distortion function depends on the structure of a system and the copula of the component lifetimes. Further, $q(\cdot)$ is continuous and increasing with $q(0)=0$ and  $q(1)=1$. For details about the distortion function, we refer to \cite{burkschat2018stochastic}. The WVE of  $T$ with PDF $f_T(\cdot)$ is defined as 
\begin{eqnarray}\label{eq4.2}
{\it VE}^x(T)&=&\int_{0}^{\infty}x^2f_T(x)[\log f_T(x)]^2dx-\bigg[\int_{0}^{\infty}xf_T(x)\log f_T(x)dx\bigg]^2\nonumber\\
&=&\int_{0}^{\infty}\phi\big(F_T(x)\big)dx-\bigg[\int_{0}^{\infty}\psi \big(F_T(x)\big)dx\bigg]^2\nonumber\\
&=&\int_{0}^{\infty}\phi\big(q(F(x))\big)dx-\bigg[\int_{0}^{\infty}\psi\big(q(F(x))\big)dx\bigg]^2\nonumber\\
&=& \int_{0}^{1}\frac{\phi\big(q(u)\big)}{f\big(F^{-1}(u)\big)}du-\bigg[\int_{0}^{1}\frac{\psi\big(q(u)\big)}{f\big(F^{-1}(u)\big)}du\bigg]^2,~\mbox{using~u=F(x)},
\end{eqnarray}
where 
\begin{eqnarray}\label{eq4.3}
\phi(u)=f(F^{-1}(u))[F^{-1}(u)\log f(F^{-1}(u))]^2
\end{eqnarray}
 and 
 \begin{eqnarray}\label{eq4.4}
 \psi(u)=F^{-1}(u)f(F^{-1}(u))[-\log f(F^{-1}(u))],
 \end{eqnarray}
 for $0\leq u\leq1$.
Next, we consider a coherent system with lifetime $T=\max\{X_1,X_2\}$, where $X_1$ and $X_2$ denote component lifetimes. It is assumed that $X_1$ and $X_2$ are independent and both follow power  distribution in the interval $(0,a)$.

\begin{example}\label{ex4.1}
	Suppose $X_1$  and $X_2$ are component lifetimes  of a coherent system with a common CDF $F(x)=(\frac{x}{a})^k,~~0<x<a,~k>0$, and consider a parallel system with lifetime $T=X_{2:2}=\max\{X_1,X_2\}$. The distortion function is $q(u)=u^2$, where $0\leq u\leq1$. Using (\ref{eq4.2}), the WVE of the coherent (parallel) system is obtained as
\begin{eqnarray}\label{eq4.5}
	\it{VE}^x(T)&=&\frac{a^2}{(2+3/k)}\bigg[\big(\log (k/a)\big)-\frac{2(1-1/k)}{(2+3/k)}\bigg]^2+\frac{4a^2(1-1/k)^2}{(2+3/k)^3}\nonumber\\
	&~&-\bigg(\frac{a}{2+1/k}\bigg)^2\bigg[\log (k/a)-\frac{2(1-1/k)}{(2+1/k)}\bigg]^2.
\end{eqnarray}
The graphs of the WVE with respect to $k$ and $a$ are plotted in Figures $12(a)$ and $12(b)$, respectively.
\end{example}

\begin{figure}[h!]
	\begin{center}
		\subfigure[]{\label{c1}\includegraphics[height=1.9in]{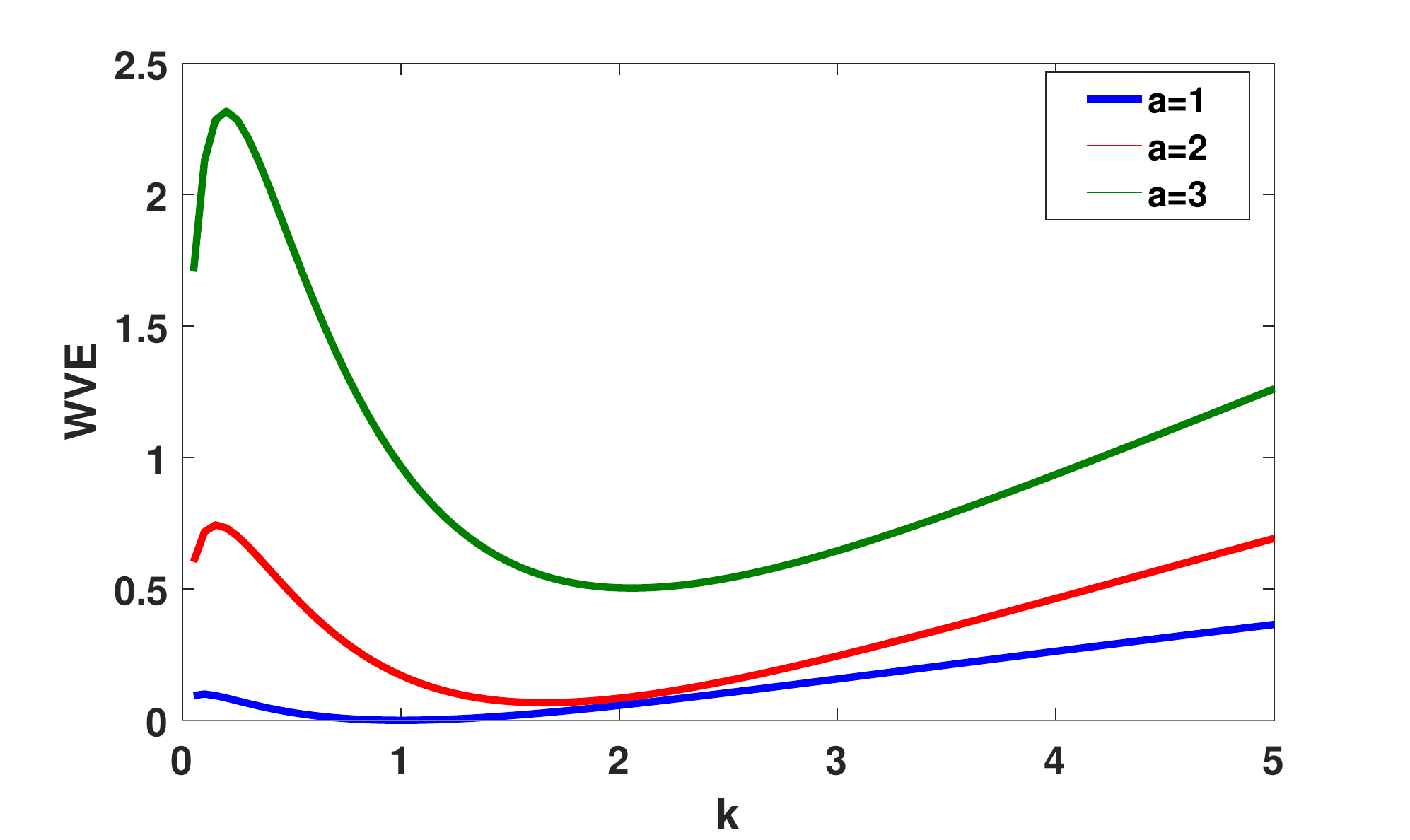}}
		\subfigure[]{\label{c1}\includegraphics[height=1.9in]{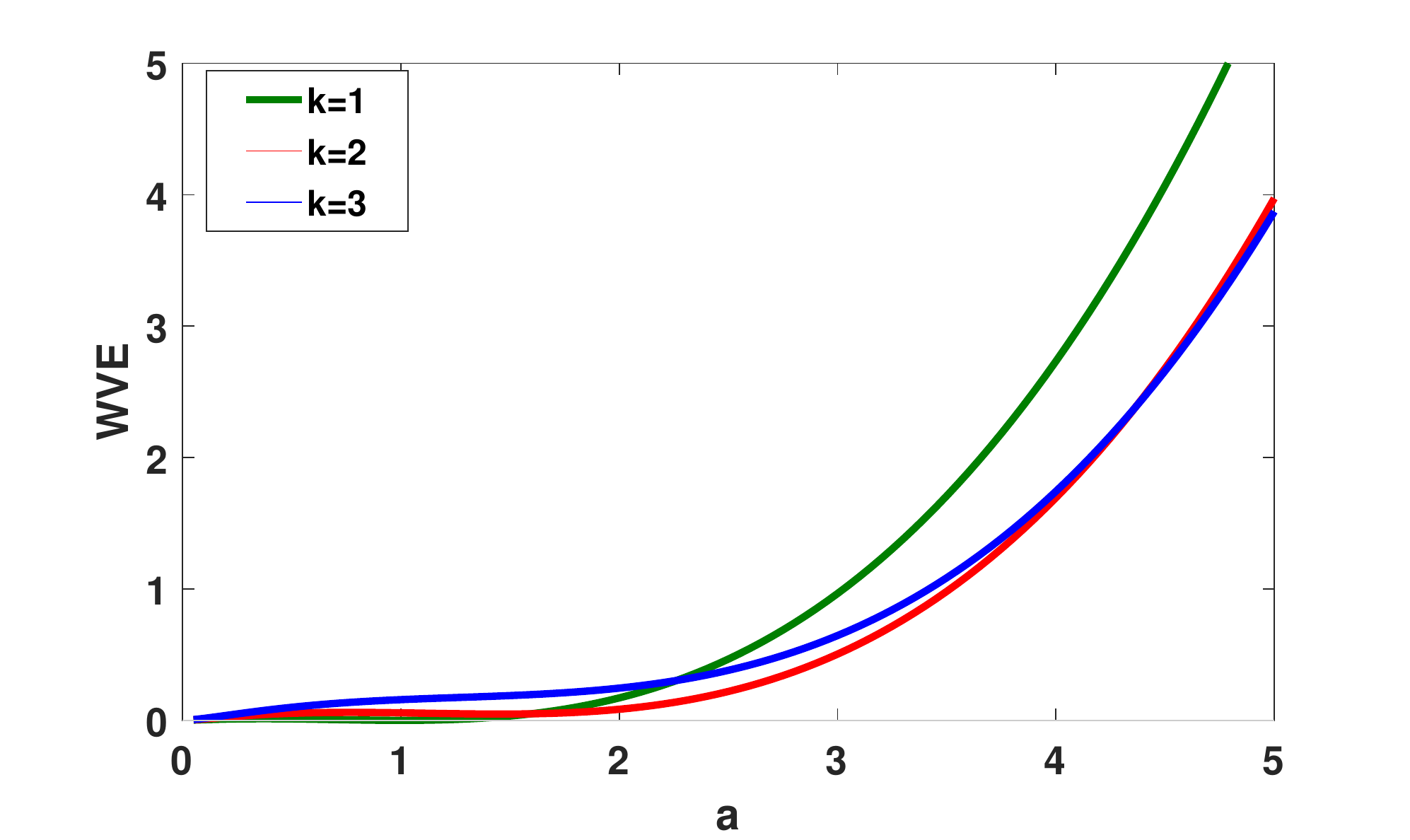}}
		\caption{Graphs of the WVE for the lifetime of a coherent system with respect to $(a)$ $k$ and $(b)$ $a$,  in Example $4.1$.}
	\end{center}
\end{figure}

Now, we obtain relation between the WVEs of the lifetimes of a coherent system and its components. 
\begin{proposition}\label{prop4.1}
	Suppose $T$ is the lifetime of a coherent system with identically distributed components. The distortion function is denoted by $q(\cdot)$. The component lifetime is denoted by $X$ with  CDF $F(\cdot)$ and PDF $f(\cdot)$. Also assume that $\phi(u)=f(F^{-1}(u))[F^{-1}(u)\log f(F^{-1}(u))]^2$ and $\psi(u)=F^{-1}(u)f(F^{-1}(u))[-\log f(F^{-1}(u))],$ for all $u\in[0,1]$. Then, if $\phi\big(q(u)\big)\geq (resp.\leq)\phi(u),$ and $\psi\big(q(u)\big)\leq (resp.\geq)\psi(u),$ for $0\le u\le1$, we have $${\it VE}^x(T)\geq (resp. \leq){\it VE}^x(X).$$
\end{proposition}	 	
\begin{proof}
	The proof is straightforward, and thus it is omitted.
\end{proof}

Next, we obtain bound of the WVE of coherent systems in terms of WSE.

\begin{proposition}
	Consider a coherent system as in Proposition \ref{prop4.1}.  Denote $\sup_{u\in[0,1]}\frac{\phi(q(u))}{\phi(u)}=\beta_{1,u},$ where $\phi(u)=f(F^{-1}(u))[F^{-1}(u)\log f(F^{-1}(u))]^2.$ Then, under the assumption in (\ref{eq2.10}), we have
	\begin{eqnarray*}
			VE^x(T)\leq \beta_{1,u} H^{w_2}(X),
		\end{eqnarray*} 
		where $H^{w_2}(X)$ is the WSE with weight $w_2(x)=\alpha x^3+\beta x^2$.
\end{proposition}
\begin{proof}
From (\ref{eq4.2}), we have
\begin{eqnarray*}
	VE^x(T)&=&\int_{0}^{1}\frac{\phi\big(q(u)\big)}{f\big(F^{-1}(u)\big)}du-\bigg[\int_{0}^{1}\frac{\psi\big(q(u)\big)}{f\big(F^{-1}(u)\big)}du\bigg]^2\\
	&=&\int_{0}^{1}\frac{\phi\big(q(u)\big)}{\phi(u)}\frac{\phi(u)}{f\big(F^{-1}(u)\big)}du-\bigg[\int_{0}^{1}\frac{\psi\big(q(u)\big)}{f\big(F^{-1}(u)\big)}du\bigg]^2\\
	&\leq& \bigg(sup_{u\in[0,1]}\frac{\phi(q(u))}{\phi(u)}\bigg)\int_{0}^{1}\frac{\phi(u)}{f\big(F^{-1}(u)\big)}du\\
	&\leq& \beta_{1,u}\int_{0}^{\infty}x^2(-\alpha x-\beta)f(x)\log f(x)dx\\
	&=&\beta_{1,u} H^{w_2}(X).
\end{eqnarray*}
Hence, the proof is completed.
\end{proof}

Below, we obtain bound of the WVE of a coherent system in terms of the WVE of component lifetime and WSE.
\begin{proposition}
	Consider a coherent system as in Proposition \ref{prop4.1}. Denote $\sup_{u\in[0,1]}\frac{\phi(q(u))}{\phi(u)}=\beta_{1,u}$, where $\phi(u)=f(F^{-1}(u))[F^{-1}(u)\log f(F^{-1}(u))]^2.$ Then,
	$${\it VE}^x(T)\leq \beta_{1,u}[{\it VE}^x(X)+(H^x(X))^2].$$
\end{proposition}

\begin{proof}
	The proof is easy, and thus it is omitted for brevity.
\end{proof}

 The following proposition provides an upper bound of ${\it VE}^x(T)$ when the PDF of the component lifetimes of a coherent system have a lower bound.
 
\begin{proposition}\label{pro4.2}
	Consider a coherent system in Proposition \ref{prop4.1}.  Let the components have  PDF $f(x)$ with support $S$, such that  
	$f(x)\geq L>0 $ for all $x\in S.$ Then,
	${\it VE}^x(T)\leq \frac{1}{L}\int^1_0\phi\big(q(u)\big)du,$ where $\phi(u)=f(F^{-1}(u))[F^{-1}(u)\log f(F^{-1}(u))]^2.$
\end{proposition} 
\begin{proof}
	Using the assumptions made, the proof readily follows from (\ref{eq4.2}).
\end{proof}

We consider the following example for the illustration of Proposition \ref{pro4.2}.
\begin{example}
	Supose $T$ indicates the lifetime of a coherent system with independent and identically distributed components following  log-uniform distribution with distribution fnction $F(x)=\frac{\log x-\log \alpha}{\log \beta-\log \alpha},~0<\alpha\le x \le \beta$. Then, $L=\frac{1}{\beta\log(\beta/\alpha)},$ and as a result from Proposition \ref{pro4.2}, we get
	$${\it VE}^x(T)\leq \beta\log\left(\frac{\beta}{\alpha}\right)\int^1_0\phi\big(q(u)\big)du,$$
	where $\phi(u)=f(F^{-1}(u))[F^{-1}(u)\log f(F^{-1}(u))]^2.$
\end{example}   

\section{WRVE for proportional hazard rate model}
In this section, we study WRVE for the proportional hazard rate (PHR) model. 
Suppose $X$ and $Y$ are two non-negative random variables with survival functions $\tilde F(\cdot)$ and $\tilde G(\cdot),$ respectively. Then, $X$ and $Y$ have PHR model if the survival function of $Y$ can be written as 
\begin{eqnarray}\label{eq5.1}
\tilde {G}(x)=[\tilde {F}(x)]^a, ~x>0,~a>0.
\end{eqnarray}
The PDF of $Y$ is
\begin{eqnarray}\label{eq5.1*}
g(x)=a[\tilde {F}(x)]^{a-1}f(x),~x>0,
\end{eqnarray}
where  $f(x)=-\frac{d}{dx}\tilde {F}(x)$ is the (baseline) PDF of $X$. The PHR model was introduced by \cite{cox1959analysis}. Later, several researchers studied the PHR model. Recently, \cite{parsa2018analysis} studied a characterization in terms of the Gini type index of PHR model.
Here, we introduce WRVE for the PHR model. The cumulative hazard rate function $\Lambda^{*}(x)$ of $Y$ is given by
\begin{eqnarray}\label{eq5.2}
\Lambda^{*}(x)=-\log \tilde{G}(x)=-\log [\tilde{F}(x)]^a=a\Lambda(x),~x>0,
\end{eqnarray}
where $\Lambda(x)=-\log \tilde F(x)$ is the cumulative hazard rate function.
The WRSE of $Y$ can be expressed as
\begin{eqnarray}\label{eq5.3}
H^x(Y;t)&=& -\int_{t}^{\infty}x\frac{g(x)}{\tilde {G}(t)}\log g(x)dx-\int_{t}^{\infty}x\frac{g(x)}{\tilde {G}(t)}\Lambda^{*}(t)dx\nonumber\\
&=& \frac{1}{[\tilde {F}(t)]^a}\bigg\{\int_{0}^{[\tilde{F}(t)]^{a}}\sigma(y:a)dy+a\Lambda(t)\int_{0}^{[\tilde{F}(t)]^{a}}\tilde {F}^{-1}(y^{1/a})dy\bigg\}\nonumber\\
&=&\frac{1}{[\tilde {F}(t)]^a}\int_{0}^{[\tilde{F}(t)]^{a}}\gamma(y:a,t)dy,
\end{eqnarray}
where $y=[\tilde{F}(x)]^a$, $\sigma(y:a)=\tilde {F}^{-1}(y^{1/a})\log \big\{ay^{1-1/a}f[\tilde {F}^{-1}(y^{1/a})]\big\} $ and 
\begin{eqnarray}\label{eq5.5**}
\gamma(y:a,t)=\tilde {F}^{-1}(y^{1/a})\log \big\{a\frac{y^{1-1/a}}{[\tilde{F}(t)]^a}f[\tilde {F}^{-1}(y^{1/a})]\big\}.
\end{eqnarray}
Note that $\tilde {F}^{-1}(y^{1/a})=\sup\{x:\tilde{F}(x)\geq y^{1/a}\}$ is known as the quantile function of $\tilde{F}(\cdot)$. In the following theorem, we evaluate the WRVE for the PHR model.
\begin{theorem}
	Let $Y$ have the survival function given by (\ref{eq5.1}). The WRVE of $Y$ is
	\begin{eqnarray}\label{eq5.4}
	{\it VE}^x(Y;t)&=&\frac{1}{[\tilde{F}(t)]^a}\int_{0}^{[\tilde{F}(t)]^{a}}[\gamma(y:a,t)]^2dy-\frac{1}{[\tilde{F}(t)]^{2a}}\bigg\{\int_{0}^{[\tilde{F}(t)]^{a}}\gamma(y:a,t)dy\bigg\}^2,
	\end{eqnarray}
	where $\gamma(y:a,t)$ is defined in (\ref{eq5.5**}).
\end{theorem}

\begin{proof}
	From (\ref{eq3.1}) and using (\ref{eq5.1}), the WRVE of $Y$ is
	\begin{eqnarray}\label{eq5.5*}
	{\it VE}^x(Y;t)=\int_{t}^{\infty}\frac{g(x)}{\tilde {G}(t)}\bigg(x\log \frac{g(x)}{\tilde {G}(t)} \bigg)^2dx-[H^x(Y;t)]^2.
	\end{eqnarray}
	Now, using $y=[\tilde{F}(x)]^a$,
	\begin{eqnarray}\label{eq5.6*}
	\int_{t}^{\infty}\frac{g(x)}{\tilde {G}(t)}\bigg(x\log \frac{g(x)}{\tilde {G}(t)} \bigg)^2dx&=&\frac{1}{[\tilde{F}(t)]^a}\bigg\{\int_{0}^{[\tilde{F}(t)]^{a}}[\sigma(y:a)]^2dy\nonumber\\
	&~&+2\Lambda^{(a)}(t)\int_{0}^{[\tilde{F}(t)]^{a}}\tilde{F}^{-1}(y^{1/a})\sigma(y:a)dy\nonumber\\
	&~& +(\Lambda^{(a)}(t))^2\int_{0}^{[\tilde{F}(t)]^{a}}[\tilde{F}^{-1}(y^{1/a})]^2dy\bigg\}\nonumber\\
	&=& \frac{1}{[\tilde{F}(t)]^a}\int_{0}^{[\tilde{F}(t)]^{a}}[\gamma(y:a,t)]^2dy.
	\end{eqnarray}
	Finally, using (\ref{eq5.3}) and (\ref{eq5.6*}) in (\ref{eq5.5*}), the theorem is established. 
\end{proof}

\begin{figure}[h!]
	\begin{center}
		\subfigure[]{\label{c1}\includegraphics[height=1.26in]{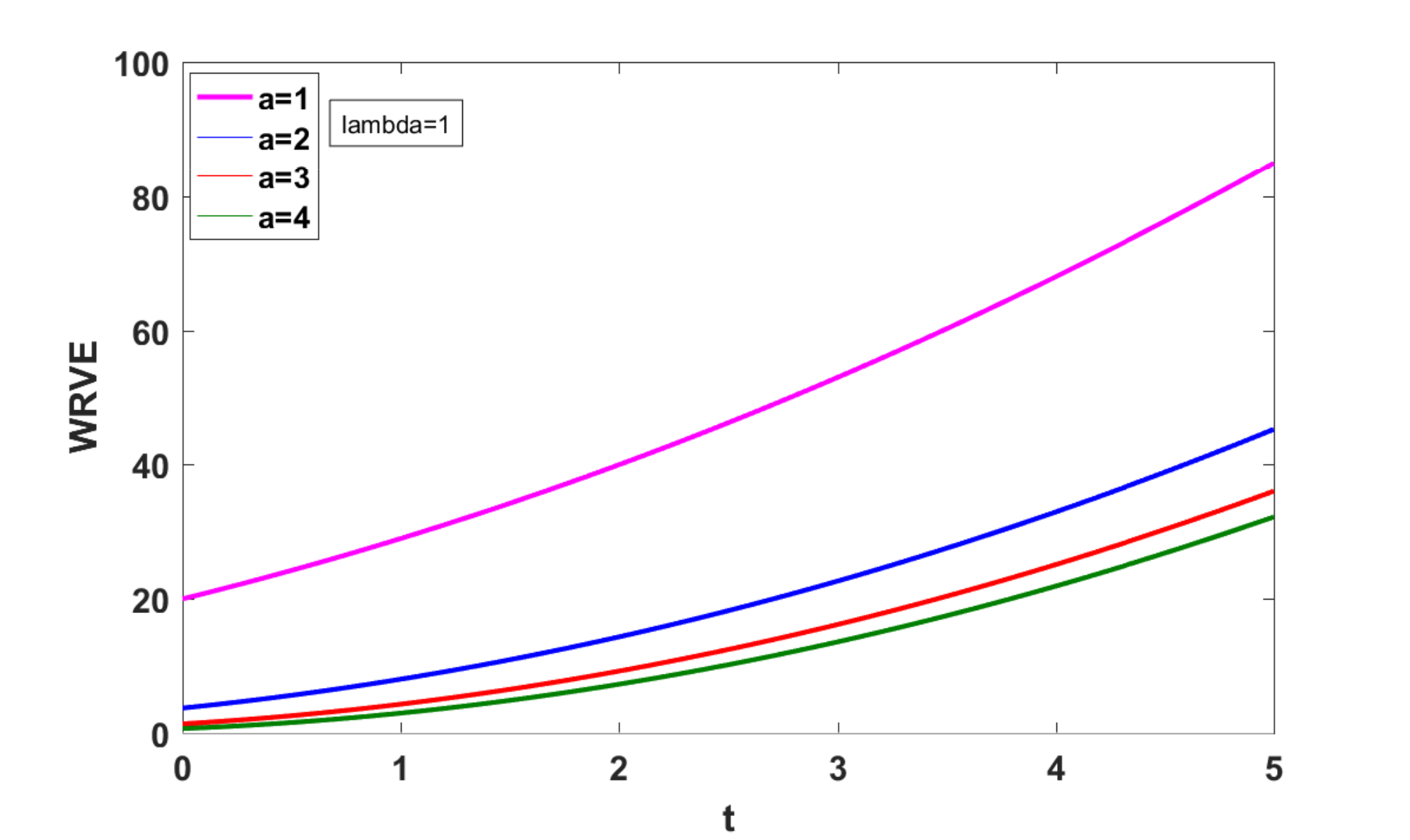}}
		\subfigure[]{\label{c1}\includegraphics[height=1.26in]{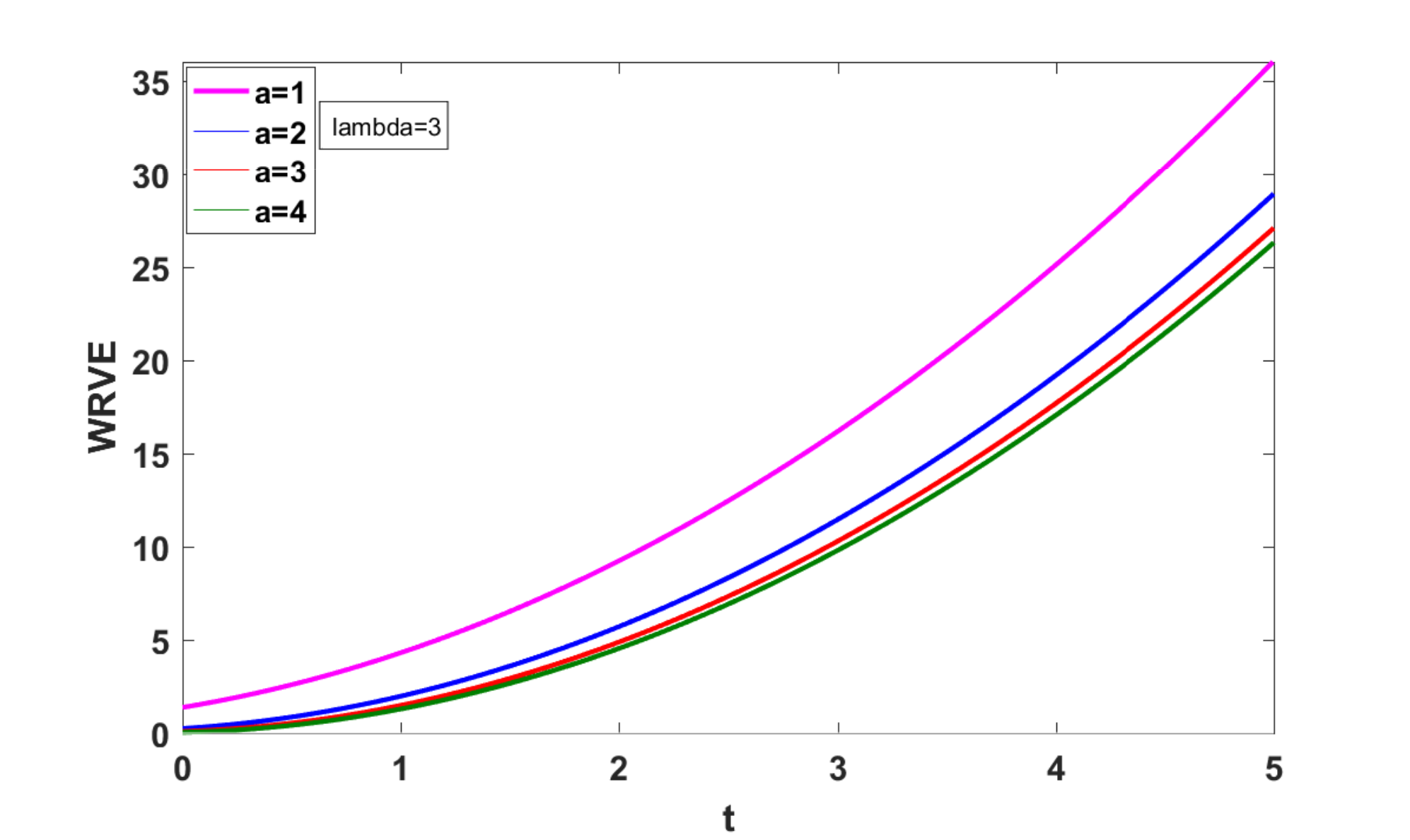}}
		\subfigure[]{\label{c1}\includegraphics[height=1.26in]{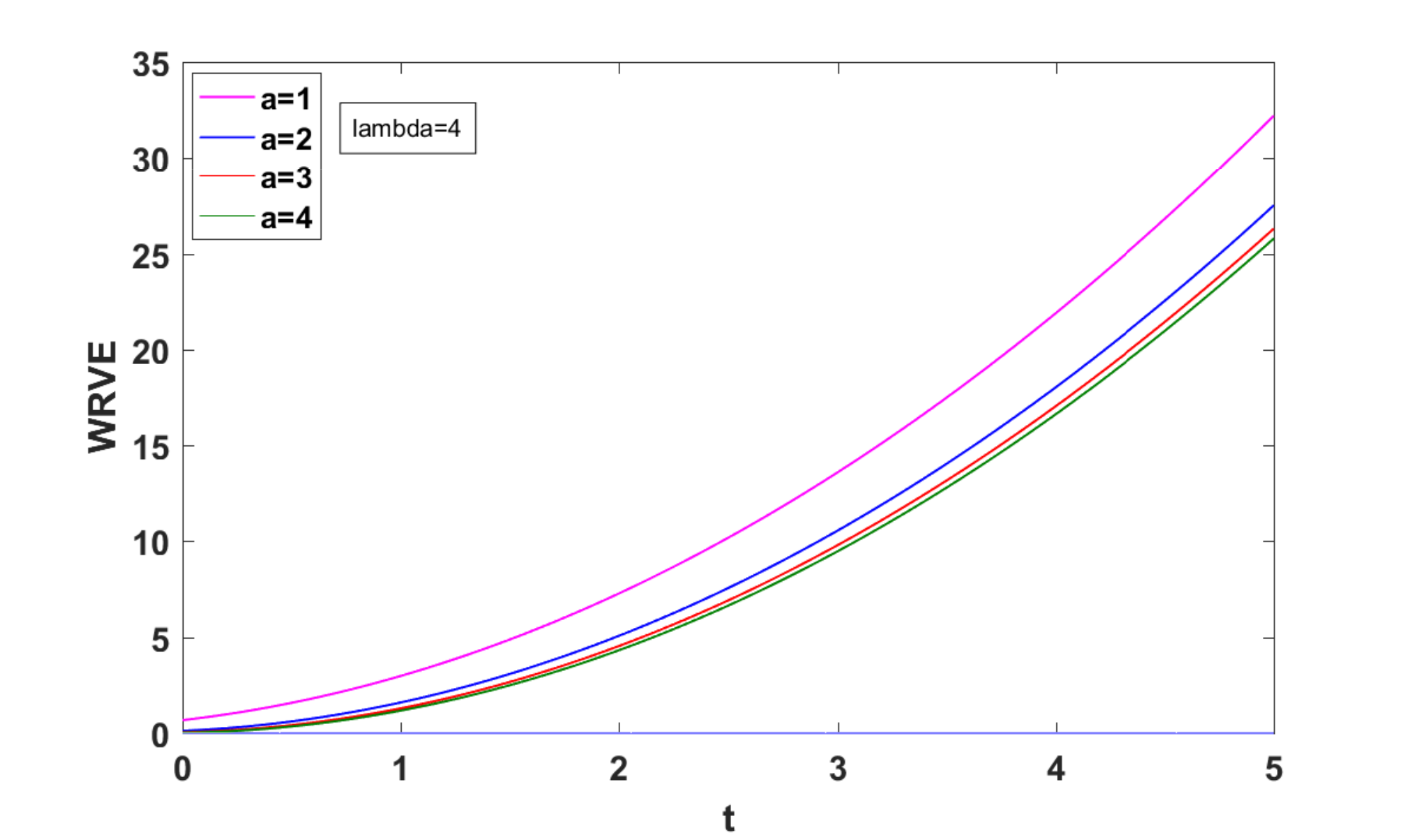}}
		\subfigure[]{\label{c1}\includegraphics[height=1.9in]{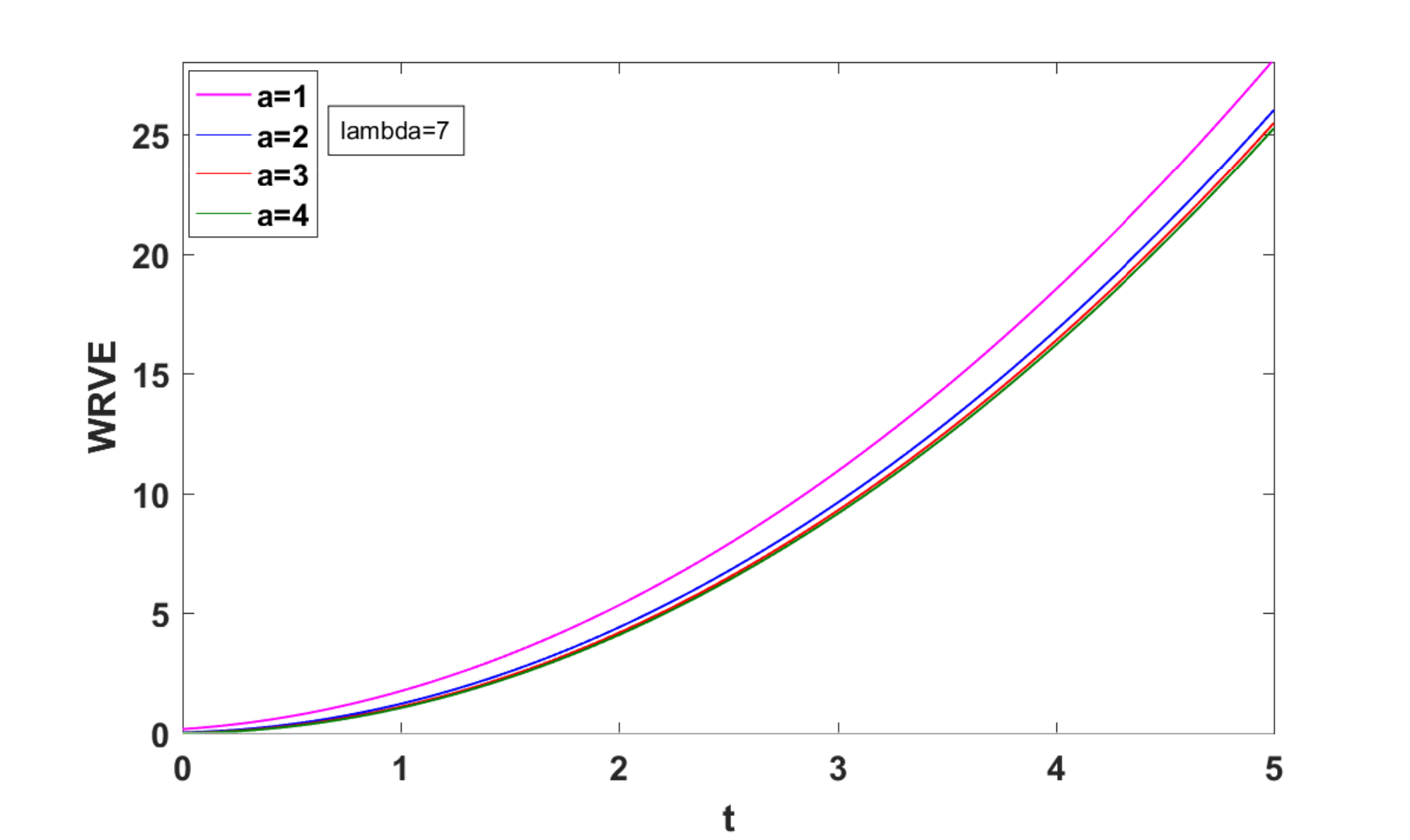}}
		\subfigure[]{\label{c1}\includegraphics[height=1.9in]{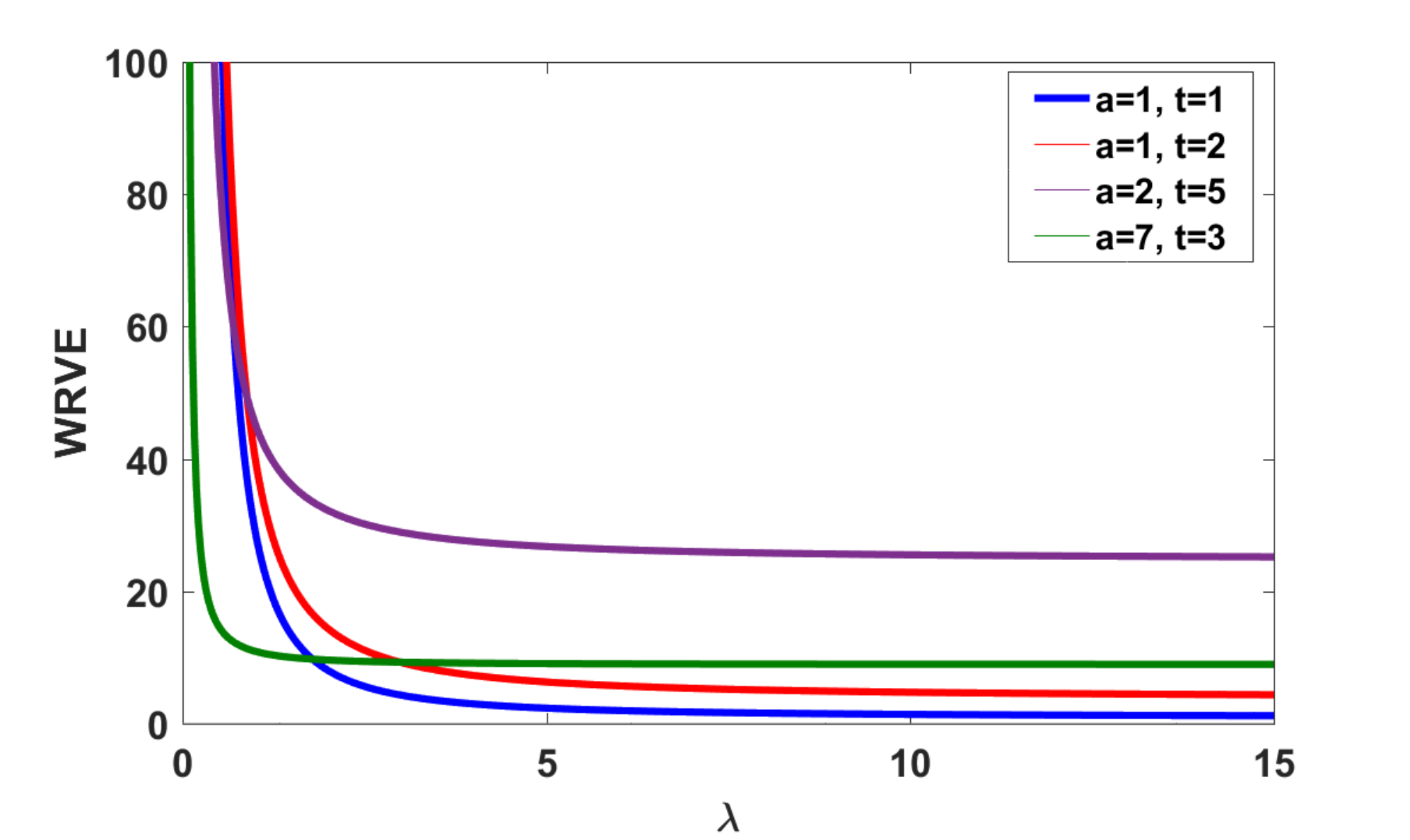}}
		\caption{Graphs of the WRVE with respect to $t$ for the system lifetime as in Example \ref{ex5.1} for $a=1,2,3,4$ when $(a)$  $\lambda=1$,  $(b)$ $\lambda=3$,  $(c)$  $\lambda=4,$ and  $(d)$  $\lambda=7$. $(e)$ Graphs of the WRVE for the system lifetime as in Example \ref{ex5.1} with respect to $\lambda,$ for different combinations of $a$ and $t$.}
	\end{center}
\end{figure}

An alternate expression of $\gamma(y:a,t)$ can be obtained by using the concept of hazard  function and cumulative hazard rate function as
\begin{eqnarray}\label{eq5.5}
\gamma(y:a,t)=\Lambda^{-1}\left(-\frac{1}{a}\log y\right)\log\left\{a y e^{a\Lambda(t)}r\left(\Lambda^{-1}(-\frac{1}{a}\log y)\right)\right\}.
\end{eqnarray}

Let a system have $n~(\ge1)$ components connected in series. The component lifetimes $X_1,\ldots,X_n$ are independent and identically distributed with a common CDF $F(t)=P(X_i\leq t)$. The survival function of the $i$th component is denoted by $\tilde {F}(t) =P(X_i > t)$. The lifetime of the series system is given by $Y = \min\{X_1,\dots,X_n\}$. It can be easily seen that the model of series system satisfies PHR model in (\ref{eq5.1}), when $a=n.$ Now, we provide an example.

\begin{example}\label{ex5.1}
	Consider a series system with component lifetimes following exponential distributions with a common CDF $F(x)=1-e^{-\lambda x}, ~ x>0,~\lambda>0$. From (\ref{eq5.5}), we obtain
	\begin{eqnarray}\label{eq5.9}
	\gamma(y:a,t)&=&-\frac{1}{a\lambda}\big[\big(a\lambda t+\log(a\lambda)\big)\log y+\big(\log y\big)^2\big].
	\end{eqnarray}
	Now, using (\ref{eq5.9}) in (\ref{eq5.4}), and after some calculations,  the closed-form WRVE of the system lifetime $Y$ can be obtained as 
	\begin{eqnarray}\label{eq5.10}
	{\it VE}^x(Y;t)&=&\frac{1}{(a\lambda)^2}\big\{[(a\lambda t+log(a\lambda))^2(a^2\lambda^2t^2+2a\lambda t+2)-2(a\lambda t+log(a\lambda))(a^3\lambda^3t^3\nonumber\\
	&~&+3a^2\lambda^2t^2+6a\lambda t+6)
	+a^4\lambda^4t^4+4a^3\lambda^3t^3+12a^2\lambda^2t^2+24a\lambda t+24]\nonumber\\
	&~&-[(1+a\lambda t)\big(a\lambda t+\log (a\lambda)\big)-(a\lambda t+1)^2-1]^2\big\}.
	\end{eqnarray}
	In Figures $13$($a$-$d$)  the WRVE in (\ref{eq5.10}) is plotted for different choices of $a$ and $\lambda$, from which it is clear that the WRVE is increasing when $t$ is increasing.  Further, the plots of  (\ref{eq5.10}) are provided in Figure $13(e)$ when the parameter $\lambda$ is varying. 
\end{example}

{\section{Non-parametric estimation of WRVE}
	In this section, we propose non-parametric estimation of the WRVE in (\ref{eq3.1}) based on the kernel density estimator of $f(\cdot)$, given by
	\begin{eqnarray}\label{eq7.1}
	\widehat f(x_i)=\frac{1}{nb_n}\sum_{i=1}^{n}K\left(\frac{x-X_i}{b_n}\right), 
	\end{eqnarray} 
	where $K(\cdot)~(\ge0)$ is known as the kernel, satisfying Lipschitz condition with $\int K(x)dx=1$. It is known that the kernel $K(\cdot)$ is symmetric with respect to the origin and $\{b_n\}$, known as bandwidths is a sequence of positive real numbers such that $b_n\rightarrow0$ and $nb_n\rightarrow\infty,$ for $n\rightarrow\infty$. For more details about kernel density estimator, the readers can see the references by \cite{rosenblatt1956remarks} and \cite{parzen1962estimation}. The non-parametric kernel estimator of the WRVE with weight $w(x)=x$ is 
	\begin{eqnarray}\label{eq7.2}
	\widehat{\it{VE}^{x}}(X;t)=\int_{t}^{\infty}\widehat{\delta}(x)(x\log  \widehat{\delta}(x))^2dx-\left[\int_{t}^{\infty}x\widehat{\delta}(x)\log  \widehat{\delta}(x)dx\right]^2;~t\ge0,
	\end{eqnarray}
	where $\widehat{\delta}(x)=\frac{\widehat f(x)}{\widehat {\tilde F}(t)}$ and $ \widehat {\tilde F}(t)=\int_{t}^{\infty}\widehat f(x)dx$. Next, we consider a simulated data set and two real data sets for illustrating the proposed non-parametric estimator in (\ref{eq7.2}).
	\subsubsection*{Simulated data set}
	We have generated a data set from the exponential distribution using Monte Carlo simulation technique. The simulation has been performed in Mathematica software. The true value of the parameter of exponential distribution is considered as $\lambda=5.5.$ For the purpose of estimation, the Gaussian kernel has been employed here, which is given by 
	\begin{eqnarray}
	K(x)=\frac{1}{\sqrt{2\pi}}e^{-\frac{x^2}{2}},
	\end{eqnarray}
	where $x$ belongs to the set of real numbers. The Bias and mean squared error (MSE) of the proposed estimator have been computed for different values of $t$ and $n$, which are presented in Table $1$. From Table $1$, we notice that the MSE of the proposed estimator decreases as the sample size $n$ increases, which guarantees the consistency of the proposed estimator.}  

\begin{table}[h!]
	\caption {The Bias and MSE for the kernel estimator of WRVE based on the simulated data set. Here, ${\it VE}^x(X;t)$ represents the true value of the WRVE.}
	\centering 
	\scalebox{.9}{\begin{tabular}{c c c c c c c c } 
			\hline\hline\vspace{.1cm} 
			$t$ &~~ $n$ &~~~~Bias &~~~ MSE & ~~${\it VE}^x(X;t)$ \\
			\hline\hline
			
			\multirow{10}{1cm} &~~50 &~~~ -0.17774 &~~~0.05540 & ~ \\[1ex]
			
			~ &~~80 &~~~ -0.18743 &~~~0.05006 & ~  \\[1ex]
			
			~0.1 &~~100 &~~~ -0.17959 &~~~0.04531 &~~0.39985  \\[1ex]
			
			~ &~~150 &~~~ -0.14185 &~~~0.03680 & ~   \\[1ex]
			
			~ &~~200 & ~~~ -0.12904 &~~~0.02887 &  ~  \\[0.5ex]
			\hline
			\multirow{10}{1cm} &~~50 & ~~~ -0.29152 &~~~0.10735 & ~ \\[1ex]
			
			~ &~~80 &~~~ -0.26200 &~~~0.08933 & ~  \\[1ex]
			
			~0.2 &~~100 & ~~~ -0.24604 &~~~0.07807 &~~0.51331  \\[1ex]
			
			~ &~~150 & ~~~ -0.22374 &~~~0.06857 & ~   \\[1ex]
			
			~ &~~200 & ~~~ -0.19449 &~~~0.05819 &  ~  \\[0.5ex]
			\hline	
			\multirow{10}{1cm} &~~50 &~~~ -0.41856 & ~~~0.20092 & ~ \\[1ex]
			
			~ &~~80 & ~~~ -0.36331 & ~~~0.16481 & ~  \\[1ex]
			
			~0.3 &~~100 & ~~~ -0.35657 &~~~0.15335 &~~0.64677  \\[1ex]
			
			~ &~~150 & ~~~ -0.30160 &~~~0.13545 & ~   \\[1ex]
			
			~ &~~200 & ~~~ -0.28645 & ~~~0.11460 &  ~  \\[.5ex]	
			
			\hline\hline	 		
	\end{tabular}} 
	
	\label{tb1} 
\end{table}
{\section*{Real data sets}
	We consider two data sets related to nano droplet dispersion and Covid-$19$ for the purpose of estimation. The nano droplet dispersion on the flat plate data set has been discussed in \cite{panahi2021adaptive}. This data set is obtained by the molecular dynamics simulation method which is used for studying the equilibrium and transport properties of nano components by Newton’s equations of motion. According to \cite{panahi2021adaptive} and \cite{dutta2022estimation}, Burr type III distribution fits the data set quite well and better than some other well-known distributions such as Weibull and inverse gamma. The data set is provided in Table $2$. The unknown parameters of the Burr type III distribution have been estimated. The estimated values of the shape parameter $\alpha=1.202347$ and scale parameter $\beta=4.701481$. Further, the bias and MSE of the estimator in (\ref{eq7.2}) have been evaluated using $600$ bootstrap samples of size $58$ using  $b_n=0.999$ in Table $3$.}
\begin{table}[h!]
	\caption {The nano droplet dispersion data set.}
	\centering 
	\scalebox{0.9}{\begin{tabular}{c c c c c c c c } 
			\hline 
			0.2289300, ~0.5810291, ~0.6935846,  ~0.7221355,  ~0.7357869,  ~0.7389012,  ~0.7486177,  ~0.7491848,\\[0.5Ex]
			0.7688918,  ~0.7689745,  ~0.7857656,  ~0.7882443,  ~0.7962973,  ~0.7972708,  ~0.8094872,  ~0.8342509,\\[0.5Ex]
			0.8451560,  ~0.8527647,  ~0.8744825,  ~0.8832821,  ~0.8905104,  ~0.8928568,  ~0.9603346, ~0.9624409,\\[0.5Ex]
			0.9677539,  ~0.9792698,  ~0.9926678,  ~1.0297182,  ~1.0890227,  ~1.0972401,  ~1.1235326,  ~1.1559192,\\[0.5Ex]
			1.1755080,  ~1.1764967,  ~1.1836366,  ~1.1975052,  ~1.2171928,  ~1.2456470,  ~1.2475189,  ~1.3245510,\\[0.5Ex]
			1.3485822,  ~1.3796668,  ~1.3932774,  ~1.4432065,  ~1.4697339,  ~1.4974976,  ~1.5356593,  ~1.5375506,\\[0.5Ex]
			1.5426158,  ~1.5430411,  ~1.5679230,  ~1.6287098,  ~1.6744305,  ~1.6838840,  ~1.7235515,  ~1.7685406,\\[0.5Ex]
			1.7980336,  ~1.8073133.\\
			\hline
	\end{tabular}} 	
	\label{tb1} 
\end{table}	
\begin{table}[h!]
	\caption {The bias, MSE  and  ${\it VE}^x(X;t)$ for the real data set in Table $2$.}
	\centering 
	\scalebox{1.4}{\begin{tabular}{c c c c c c c c } 
			\hline\hline\vspace{.1cm} 
			~~~$t$ &~~~~Bias &~~~ MSE & ~~${\it VE}^x(X;t)$ \\
			\hline\hline
			
			~0.01 &~~~-0.39998 &~~~0.18895&~~~7.65442  \\[1ex]
			
			~0.05 &~~~-0.39147 &~~~0.18287&~~~7.65442 \\[1ex]
			
			~0.10 &~~~-0.38704 &~~~0.18229&~~~7.65443   \\[1ex]
			
			~0.20 &~~~-0.37023 &~~~0.16924 &~~~7.65503    \\[1ex]
			
			~0.30 &~~~-0.35943 &~~~0.16094 &~~~7.66043   \\[0.5ex]
			~0.40 &~~~-0.34001 &~~~0.14839 &~~~7.68493   \\[0.5ex]
			~0.50 &~~~-0.36229 &~~~0.16217 &~~~7.76122   \\[0.5ex]
			~0.60 &~~-0.49115 &~~~0.27427 &~~~7.94443   \\[0.5ex]
			~0.70 &~~-0.77791 & ~~~0.63932 &~~~8.30439   \\[0.5ex]

			\hline\hline	 		
	\end{tabular}} 
	
\end{table}

In this part of the section, we consider a Covid-$19$ data set of France reported by \cite{alrumayh2023optimal} containing the daily death rate for 24 days collected from October $1$ to October $24$, $2021$. The data set is provided in Table $\ref{tb1}$. Recently, \cite{dutta2023bayesian} studied the goodness-of-fit test of this data set. The authors  mentioned that the data set fits with  the logistic exponential distribution, with shape parameter $\alpha=2.4719$ and scale parameter $\lambda=1.7619$. Similar to the previous data set, here the bias and MSE for the proposed estimator and values of $\it{VE}^x(X;t)$ are presented in Table \ref{tb2}. We assume the value of $b_n=0.45$ for estimation purposes.
\begin{table}[h!]
	\caption {Covid-19 data set of France  containing the daily death rate.}
	\centering 
	\scalebox{1}{\begin{tabular}{c c c c c c c c } 
			\hline
			0.0740, ~0.1190, ~0.1344, ~0.1926, ~0.2232, ~0.3140, ~0.3243, ~0.3393, ~0.3563, ~0.3706,\\
			~0.3843, ~0.4164, ~0.4482, ~0.4578, ~0.4616, ~0.4755, ~0.4917, ~0.5045, ~0.5069, ~0.5325,\\
			~0.5625, ~0.5972, ~0.8057, ~0.8078.\\
			\hline	 		
	\end{tabular}} 
	\label{tb1} 
\end{table}

\begin{table}[h!]
	\caption {The bias, MSE  and  ${\it VE}^x(X;t)$ for the real data set in Table $4$.}
	\centering 
	\scalebox{1.4}{\begin{tabular}{c c c c c c c c } 
			\hline\hline\vspace{.1cm} 
			~~~$t$ &~~~~Bias &~~~ MSE & ~~$\it{VE}^x(X;t)$ \\
			\hline\hline
			
			
			~0.01 &~~~0.21895 &~~~0.05103&~~~0.61632 \\[1ex]
			
			~0.05 &~~~0.22029 &~~~0.05171&~~~0.61733  \\[1ex]
			~0.10&~~0.22024 &~~~0.05195 &~~~0.62262   \\[0.5ex]
			
			~0.20 &~~~0.21703 &~~~0.05050 &~~~0.65976   \\[1ex]
			~0.25 &~~~0.18425 &~~~0.03773 &~~~0.69855  \\[1ex]
			~0.30&~~0.12787 &~~~0.02024 &~~~0.75391  \\[0.5ex]
			~0.35&~~0.10462 &~~~0.01523 &~~~0.82615 \\[0.5ex]	
			~0.40 &~~~0.10176 &~~~0.01813 &~~~0.91412   \\[0.5ex]
			~0.45 &~~~0.07619 &~~~0.01650 &~~~1.01581   \\[0.5ex]

			\hline\hline	 		
	\end{tabular}} 
	\label{tb2} 
\end{table}

\section{Conclusions} 
In this paper, we have introduced the concept of WVE and WRVE. Various properties of these two variability measures have been obtained. The effect under affine transformations and some bounds have been obtained. A justification of the usefullness of WRVE over the RVE has been illustrated. Further, the WVE has been studied for the coherent systems. The sufficient conditions, under which the WVE of a system lifetime is smaller/larger than that of its component lifetimes have been derived. The WRVE has been also studied for PHR models. A non-parametric estimator of the WRVE has been proposed, and then the estimator has been illustrated using a simulated data set and two real-life data sets. Finally, we mention here that the proposed measures could be introduced in a more general context omitting the non-negativity assumption. However, its use and interpretation may need to be examined carefully.

\section*{Acknowledgements}  
{Both the authors thank the referees for helpful comments and suggestions which lead to the substantial improvements.} The author Shital Saha thanks the UGC (Award No. $191620139416$), India, for the financial assistantship received to carry out this research work. Both authors thank the research facilities provided by the Department of Mathematics, National Institute of Technology Rourkela, Odisha, India.

\section*{Conflicts of interest} The authors declare no conflict of interest.

\bibliography{refference}
\end{document}